\documentclass[a4paper,11pt]{article}

\usepackage{amsmath}
\usepackage{amssymb}
\usepackage{amsthm}
\usepackage{graphicx}
\usepackage{enumerate}
\usepackage[dvipsnames]{xcolor}
\usepackage{hyperref}
\usepackage[pagewise]{lineno}

\newtheorem{theorem}{Theorem}[section]
\newtheorem{proposition}[theorem]{Proposition}

\newtheorem{assumption}{Assumption A}

\theoremstyle{definition}

\numberwithin{equation}{section}

\begin{document}
	
\title{ AN-SPS: Adaptive Sample Size Nonmonotone Line Search Spectral Projected  Subgradient Method for  Convex  Constrained Optimization Problems}
\author{Nata\v sa Krklec Jerinki\' c\footnote{Department of Mathematics and Informatics, Faculty of Sciences, University of Novi Sad, Trg Dositeja Obradovi\' ca 4, 21000 Novi Sad, Serbia. e-mail: \texttt{natasa.krklec@dmi.uns.ac.rs}}, Tijana Ostoji\' c \footnote{Department of Fundamental Sciences, Faculty of Technical Sciences, University of Novi
Sad, Trg Dositeja Obradovi\' ca 6, 21000 Novi Sad, Serbia. e-mail: \texttt{tijana.ostojic@uns.ac.rs}} \footnote{Corresponding author} }
\date{October 10, 2023}
\maketitle

\begin{abstract}
We consider convex optimization problems with a possibly nonsmooth objective function in the form of a mathematical expectation. The proposed framework (AN-SPS) employs Sample Average Approximations (SAA) to approximate the objective function, which is either unavailable or too costly to compute. The sample size is chosen in an adaptive manner, which eventually pushes the SAA error to zero almost surely (a.s.). The search direction is based on a scaled subgradient and a spectral coefficient, both related to the SAA function. The step size is obtained via a nonmonotone line search over a predefined interval, which yields a theoretically sound and practically efficient algorithm. The method retains feasibility by projecting the resulting points onto a feasible set. The a.s. convergence of AN-SPS method is proved without the assumption of a bounded feasible set or bounded iterates. Preliminary numerical results on Hinge loss problems reveal the advantages of the proposed adaptive scheme. In addition, a study of different nonmonotone line search strategies in combination with different spectral coefficients within AN-SPS framework is also conducted, yielding some hints for future work.
\end{abstract}


\section{Introduction}\label{sec1}

{\bf{The problem.}} We consider convex constrained optimization problem with the objective function in the form of mathematical expectation, i.e.,  

\begin{equation} 
\min_{x\in\Omega} f(x)=E(\tilde{f}(x,\xi)),
\label{prob1}
\end{equation}
where $\Omega\subset\mathbb{R}^n$ is a convex set,  $\tilde{f}:\mathbb{R}^n \times \mathbb{R}^d \rightarrow \mathbb{R}$ is continuous and convex function with respect to $x$, $\xi:\mathcal{A}\rightarrow \mathbb{R}^d$ is random vector on a probability space $(\mathcal{A}, \mathcal{F}, P)$  and $f$ is continuous and bounded from below on $\Omega$. We assume that it is possible to find an exact projection onto the feasible set, so a typical representative of $\Omega$ is $n$-dimensional ball, nonnegativity constraints, or generic box constraints. We do not impose smoothness of $\tilde{f}$, so we are dealing with nondifferentiable functions $\tilde{f}$ in general.   This framework covers many important optimization problems, \cite{cevher, LSOS26, LSOS27, LSOS38}, such as Hinge loss within a machine learning framework. Moreover, it is known that general constrained optimization problems may be solved through penalty methods, where the relevant subproblems are often transformed into nonnegativity-constrained problems by introducing slack variables or semi-smooth unconstrained problems. Both cases fall into the framework that we consider, provided that the objective function is convex. 

{\bf{Variable sample size schemes.}} The objective function in \eqref{prob1} is usually unavailable or too costly to be evaluated directly \cite{shapiro}. For instance, there are many applications where the analytical form of the mathematical expectation cannot be attained. Moreover, there are also online training problems (e.g., optimization problems that come from time series analysis) where the sample size grows as time goes by. However, even if the sample size is finite and we are dealing with a finite sum problem, working with the full sample throughout the whole optimization process is usually too costly or, moreover,  unnecessarily. This is the reason why  Variable Sample Size (VSS) schemes have been developed over the past few decades overlapping with the Big Data era \cite{BCT, SBNKNKJ, Heur1, HT2, NKNKJ2, Heur3, NKJMM, andrea}, to name just a few.  The idea is to work with Sample Average Approximation (SAA) functions \begin{equation} 
f_{\mathcal{N}}(x)=\frac{1}{N}\sum_{i\in\mathcal{N}}f_i(x),
\label{SAA}
\end{equation}
where $ f_i(x)=\tilde{f}(x,\xi_i) $ and $\xi_i, i=1,2,...$ are usually assumed to be independent and identically distributed (i.i.d.) \cite{shapiro}.
$ N=|\mathcal{N}| $ determines the size of a sample used for the approximation and it is varied across the iterations, allowing cheaper approximations whenever possible.

{\bf{Nonmonotone line search.}} Line search methods are known as a  powerful tool in classical optimization, especially in smooth deterministic case. They provide global convergence with a good practical performance. However, in a stochastic nonsmooth framework, it is very hard to analyze them. In the stochastic case, line search yields biased estimators of the function values in subsequent iteration points, which complicates classical analysis and seeks alternative approaches \cite{LSOS, iusem, kiwiel, KLOS, review}. In the nonsmooth framework, even if strong convexity holds, the lower bounding of the step size is very hard. In \cite{kiwiel}, the steps are bounded from below, but not uniformly since they depend on the tolerance parameter, which tends to zero if convergence towards the optimal point is aimed instead of a nearly-optimal point.  On the other hand, a predefined step size sequence such as the harmonic one is enough to guarantee the (a.s.)  convergence under the standard assumptions \cite{boyd, nedic}, even in the mini-batch or SA (Stochastic Approximation) framework \cite{robbins, spall}.  Unfortunately, this choice usually yields very slow convergence in practice \cite{boyd}. SPS (Spectral Projected Subgradient) framework \cite{nasdrugi} proposes a combination of line search and predefined sequence by performing the line search on predefined intervals, keeping the method both fast and theoretically sound. 

Classical Armijo line search needs descent direction in order to be well defined. While in smooth optimization it is easy to determine it, in the nonsmooth case it is a much more challenging task \cite{kiwiel, kineski}. Moreover, allowing more freedom for the step size selection may be beneficial, especially when the search directions are of spectral type \cite{Birgin1, nasdrugi, loreto2}. Finally, having in mind that  VSS schemes work with approximate functions, nonmonotone line search seems like a reasonable choice in this setup. 

{\bf{Spectral coefficients.}} Although the considered problem  \eqref{prob1} is not smooth, including some second-order information seems to be beneficial according to the existing results \cite{nasprvi, loreto, kineski}. Moreover, spectral-like methods proved to be efficient in the stochastic framework with increasing accuracy \cite{greta, NKNKJ}. 
We present a framework that allows different spectral coefficients to be combined with subgradient directions. Following \cite{diSerafino}, we test different choices of Barzilai-Borwein (BB) rules in a stochastic environment. 

One of the key points lies in an adaptive sample size strategy. Roughly speaking, the main idea is to balance two types of errors - the one that measures how far is the iterate from the current SAA function's constrained optimum, and the one that estimates the SAA error. More precisely, we present an adaptive strategy that determines when to switch to the next level of accuracy and prove that this pushes the sample size to infinity (or to the full sample size in a finite sum case). In the SPS framework, the convergence result was proved under the assumption of the sample size increase at each iteration, while for AN-SPS the increase is a consequence of the algorithm's construction rather than the assumption. 

We believe that one more important advantage with respect to SPS is a proposed scaling of the subgradient direction. The scaling strategy is not new in general \cite{burke2}, but it is a novelty with respect to the SPS framework. One of the most important consequences of this modification is that the convergence result is proved without boundedness assumptions - we do not impose any assumption of uniformly bounded subgradients, feasible set, nor iterates. Instead, we prove that AN-SPS provides the bounded sequence of iterates under a  mild sample size growth condition. 

The main result - almost sure convergence of the whole sequence of iterates - is proved under rather standard conditions for stochastic analysis.  Moreover, in the finite sum problem case, the convergence is deterministic, and it is proved under a significantly reduced set of assumptions with respect to the general case \eqref{prob1}.  
Furthermore, we proved that the worst-case complexity can achieve the order of $\varepsilon^{-1}$. Although the worst-case complexity result stated in Theorem \ref{teocompl} is comparable to the complexity of standard subgradient methods with a predefined step size sequence and its stochastic variant  (both of order  $\varepsilon^{-2}$, see   \cite{bubi, nemirovsky} for instance), we believe that the advantage of the proposed method lies in its ability to accept larger steps and employ spectral coefficients combined with a nonmonotone line search, which can significantly speed up the method. Furthermore, the proposed method provides a wide framework for improving computational cost complexity since it allows different sampling strategies to be employed. Preliminary numerical tests on Hinge loss problems and common data sets for machine learning show the advantages of the proposed adaptive VSS strategy. We also present the results of a study that investigates how different spectral coefficients combine with different nonmonotone rules. 

{\bf{Contributions.}}  This paper may be seen as a continuation of the work presented in \cite{nasdrugi} and further development of the algorithm LS-SPS (Line Search Spectral Projected Subgradient Method for Nonsmooth Optimization) proposed therein. In this light, the main contributions of this work are the following: 
\begin{itemize}
\item[i)] An adaptive sample size strategy is proposed and we prove  that this strategy pushes the sample size to infinity (or to the maximal sample size in  the finite sum case); 

\item[ii)] We show that the scaling can relax the boundedness assumptions on subgradients, iterates, and feasible set; 

\item[iii)]  For finite sum problems, we provide the worst-case complexity analysis of the proposed method;

\item[iv)] The LS-SPS is generalized in the sense that we allow different nonmonotone line search rules. Although  important for the practical behavior of the algorithm, this change does not affect the convergence analysis and it is investigated mainly through numerical experiments; 

\item[v)] Considering the spectral coefficients, we investigate different strategies for its formulation \cite{diSerafino} in a stochastic framework. 
 Different combinations of spectral coefficients and nonmonotone rules are evaluated within numerical experiments conducted on machine learning Hinge loss problems.
\end{itemize}

{\bf{Paper organization.}} The algorithm is presented in Section \ref{sec2}. Convergence analysis is conducted in Section \ref{sec3}, while preliminary numerical results are reported in Section \ref{sec4}. Section \ref{sec5} is devoted to the conclusions and some proofs are delegated to the Appendix (Section \ref{sec6}). 

{\bf{Notation.}} The notation we use is the following. Vector $x \in \mathbb{R}^n$ is considered as a column vector.  $||\cdot||$ represents the Euclidean norm.  $x_k$ represents an iterate, i.e., an approximation of a solution of problem \eqref{prob1} at iteration $k$. The sample used to approximate the objective function via \eqref{SAA} at iteration $k$ is denoted by  $\mathcal{N}_k$, while   $N_k$ denotes its cardinality. The exact orthogonal projection of $x\in \mathbb{R}^n$ onto $\Omega$ will be denoted as $P_{\Omega}(x)$.  $X^*$ and $f^*$ denote a set of solutions and an optimal value of problem  \eqref{prob1}, respectively. We denote a solution of the problem \eqref{prob1} by $x^*$. Analogously, we denote by  $x_{\cal{N}}^*$, $X_{\cal{N}}^*$ and $f_{\cal{N}}^*$ a solution, set of all solutions and an optimal value of an approximate problem $\min_{x \in \Omega}  f_{\cal{N}}(x)$, respectively. Relevant constants are denoted by capital $C$ (e.g., $C_1$), underlined letter  (e.g., $\underline{\zeta}$) or  overlined letter (e.g., $\bar{c}_1$). We denote by  $\bar{e}_k=|f_{\mathcal{N}_k}(x_k)-f(x_k)|+|f_{\mathcal{N}_k}(x^*)-f(x^*)|$ the relevant SAA errors at iteration $k$.

\section{The Method}\label{sec2}

In this section, we state the proposed  AN-SPS framework algorithm. 
In order to define the rule for updating the sample size  $N_k=|{\cal{N}}_k|$, we introduce the SAA error measure $h(N_k)$, i.e., a proxy for $ |f(x)-f_{\mathcal{N}_k}(x)|$, as follows. 
  In the finite sum case with the  full sample size $N_{max}<\infty$ we define 
  $$h(N_k)=\frac{N_{max}-N_k}{N_{max}},$$
  while in general (unbounded sample size) case we define
  $$h(N_k)=\frac{1}{N_k}.$$ 
  Notice that in both cases we have $h: \mathbb{N} \to [0,1]$  which is monotonically decreasing and strictly positive if the full sample is not attained. Moreover, in the finite sum case we have $h(N_k)=0$ if and only if $N_k=N_{max}$, while in unbounded sample case  we have  $\lim_{N_k \rightarrow\infty}h(N_k)= 0.$
 Other choices are eligible as well, but we keep these ones for simplicity. 

Let us define the upper bound of the predefined interval for the line search by  
$$ \bar{\alpha}_k=\min \{1,C_2/k\},$$ 
where $C_2>0$ can be arbitrarily large.  

\noindent {\bf Algorithm 1: AN-SPS} ({\bf A}daptive Sample Size  {\bf N}onmonotone Line Search  {\bf S}pectral {\bf P}rojected {\bf S}ubgradient Method)
\label{SPGNS}
\begin{itemize}
\item[\textbf{S0}] \textit{Initialization.} Given $N_0, m \in \mathbb{N},$ $x_0\in\Omega,$ $ C_2>0,$  $0 <  \underline{\zeta}\leq  \overline{\zeta}< \infty,$ $\zeta_0 \in \left[\underline{\zeta},\overline{\zeta}\right].$ Set $k = 0$ and $F_0=f_{\mathcal{N}_{0}}(x_0)$.

\item[\textbf{S1}] \textit{Search direction.} Choose  $\bar{g}_{k}\in  \partial f_{\mathcal{N}_{k}}(x_k)$. Set $q_k=\max \{1,\|\bar{g}_{k}\| \}$,  $v_k=\bar{g}_{k}/q_k$ and   $p_{k}=-\zeta_k v_{k}$. 

\item[\textbf{S2}] \textit{Step size.}  

  \textbf{If} $k=0$, set  $\alpha_0=1$. \\
\textbf{Else,} choose $m$ points $\{\tilde{\alpha}_k^1,...,\tilde{\alpha}_k^m\}$ such that 
  $$\frac{1}{k}<\tilde{\alpha}_k^1<\tilde{\alpha}_k^2<\ldots<\tilde{\alpha}_k^m =\bar{\alpha}_k.$$
     \begin{itemize}
         \item[] \textbf{If} the condition 
            \begin{equation}
               f_{\mathcal{N}_{k}}(x_k+\tilde{\alpha}_k^j  p_{k})\leq F_k-\eta \tilde{\alpha}_k^j ||p_{k}||^2
               \label{Armijo}
             \end{equation} 
             is satisfied for some $j \in \{m,m-1,\ldots,1\}$, set $\alpha_k=\tilde{\alpha}_k^{j}$ with the largest possible $j$.\\ 
                  \textbf{Else,} set  $\alpha_k=\frac{1}{k}.$
\end{itemize}
 
\item[\textbf{S3}] \textit{Main update.} Set $z_{k+1}=x_k+\alpha_k  p_{k}$, $x_{k+1} = P_{\Omega}(z_{k+1})$,   $s_k = x_{k+1}-x_k$ and $\theta_k=\|s_k\|.$

\item[\textbf{S4}] \textit{Spectral coefficient update.}  Choose $\zeta_{k+1} \in [\underline{\zeta}, \bar{\zeta}].$ 

\item[\textbf{S5}]  \textit{Sample size update.} If 
$\theta_k < h(N_k)$, choose  $N_{k+1}> N_k$ and a new sample $\mathcal{N}_{k+1}$. Else,   $\mathcal{N}_{k+1}=\mathcal{N}_{k}$.  

\item[\textbf{S6}] \textit{Nonmonotone line search     update.} Determine $F_{k+1}$ such that  $$f_{\mathcal{N}_{k+1}}(x_{k+1}) \leq F_{k+1} < \infty.$$ 
\item[\textbf{S7}]  \textit{Iteration update.} Set $k:=k+1$ and go to \textbf{S1}. 
\end{itemize}

First, notice that the initialization and Step S3 ensure the feasibility of the iterates. In Step S1, we choose an arbitrary subgradient of the current approximation function $f_{\mathcal{N}_{k}}$ at point $x_k$. Further, scaling with $q_k$ implies that   $\|v_k\| \leq 1$. Moreover, the boundedness of the spectral coefficient $\zeta_k$ yields uniformly bounded search directions $p_k$. This is very important from the theoretical point of view since it helps us to overcome the boundedness assumptions mentioned in the Introduction. 

For the step size selection, we practically use a backtracking-type procedure over the predefined interval $(\frac{1}{k},\bar{\alpha}_k]$. Notice that $C_2$ can be arbitrarily large so that in practice $\bar{\alpha}_k$ is equal to 1 in most of the iterations. However, the upper bound $C_2/k$ is needed to ensure theoretical convergence results. The lower bound, $1/k$, is known as a good choice from the theoretical point of view, and often a bad choice in practice. Thus, roughly speaking, the line search checks if larger, but still theoretically sound steps are eligible. Since the Armijo-like condition \eqref{Armijo} is checked in at most $m$ points, the procedure is well defined since if none of these candidate points satisfies condition \eqref{Armijo}, the step size is set to  $1/k$. 
This allows us to use nondescent directions and practically arbitrary nonmonotone (or monotone) rule determined by the choice of $F_k$. For instance, $F_k$ can be set to $f_{{\cal{N}}_{k}}(x_k)+0.5^k$, but various other choices are possible as well. The choice of nonmonotone rule does not affect the theoretical convergence of the algorithm, but it can be very important in practice as we will show in the Numerical results section. Parameter $m$ influences the per-iteration cost of the algorithm since it upper bounds the number of the function $f_{{\cal{N}}_k}$ evaluations within one line search procedure, i.e., within one iteration. Having in mind that the function $f_{{\cal{N}}_k}$ is just an estimate of the objective function in general, we suggest that $m$ should be relatively small in order to avoid an unnecessarily precise line search and high computational costs. On the other hand, having $m$ too small may yield smaller step sizes since $1/k$ is more likely to be accepted in general. Numerical results presented in Section \ref{sec4} are obtained by taking $m=2$ in all conducted experiments. However, tuning this parameter or even making it adaptive may be an interesting topic to investigate.

We will test the performance of some choices for the spectral coefficients, where, from a theoretical point of view, the only requirement is the safeguard stated in Step S4 of the algorithm - $\zeta_k$ must remain within a positive, bounded interval $[\underline{\zeta}, \bar{\zeta}]$.  

Finally, the adaptive sample size strategy is determined within Step S5. The overall step length $\theta_k$  may be considered as a measure of stationarity related to the current objective function approximation $f_{\mathcal{N}_{k}}$. In particular, we will show that, if the sample size is fixed,  $\theta_k$ tends to zero and the sequence of iterates is approaching a minimizer of the current SAA function  (see the proof of Theorem \ref{teoNk} in the sequel). When $\theta_k$ is relatively small (smaller than the measure of SAA error $h(N_k)$), we decide that the two errors are in balance and that we should improve the level of accuracy by enlarging the sample.  Notice that Step S5 allows a completely different sample $\mathcal{N}_{k+1}$ in general with respect to $\mathcal{N}_{k}$ in the case when the sample size is increased. 
However, if the sample size is unchanged, the sample is unchanged, i.e., $\mathcal{N}_{k+1}=\mathcal{N}_{k}$, which allows non-cumulative samples to fit within the proposed framework as well. 

AN-SPS algorithm detects the iteration within which the sample size needs to be increased, but it allows full freedom in the choice of the subsequent sample size as long as it is larger than the current one. After some preliminary tests, we end up with the following  selection: when the sample size is increased, it is done as 
\begin{equation} 
\label{increase} 
N_{k+1}=\lceil \max\{ (1+\theta_k)N_k, r N_k\} \rceil, \end{equation}   
with $r=1.1$. Although some other choices such as direct balancing of $\theta_k$ and $h(N_{k+1})$  seemed more intuitive, they were all outperformed by the choice \eqref{increase}. Disregarding the safeguard part where, in case of $\theta_k=0$, the sample size is increased by 10$\%$, the relation becomes 
$$1+\frac{N_{k+1}-N_k}{N_k}\approx 1+\theta_k.$$
Thus, the relative increase in the sample size is balanced with the stationarity measure. Furthermore, since we know that in these iterations $\theta_k<h(N_k)$, we obtain the relative increase bounded above by  $h(N_k)$. Apparently, this helps the algorithm to overcome the problems caused by the non-beneficiary fast growth of the sample size. 

\section{Convergence analysis}\label{sec3}

This section is devoted to the convergence analysis of the proposed method. One of the main results lies in Theorem \ref{teoNk} where we prove that $h(N_k)$ tends to zero. This means that the sample size tends to infinity in the unbounded sample case,  while in the finite sum case, it means that the full sample is eventually reached.  After that, we show that we can relax the common assumption of uniformly bounded subgradients stated in the convergence analysis in \cite{nasdrugi}. Normalized subgradients have been used in the literature, but they represent a novelty with respect to the SPS framework.   Hence, we need to show that this kind of scaling does not deteriorate the relevant convergence results. 
We state the boundedness of iterates within Proposition \ref{teoBounded}.  Although the convergence result stated in Theorem \ref{teoAS} mainly follows from the analysis of SPS \cite{nasdrugi} (see Theorem 3.1 therein), we provide the proof in the Appendix (Section \ref{sec6}) since it is based on different foundations. Therefore, we show that  AN-SPS retains almost sure convergence under relaxed assumptions with respect to LS-SPS proposed in \cite{nasdrugi}, while, on the other hand,  it brings more freedom to the choice of nonmonotone line search and the spectral coefficient. Finally, we formalize the conditions needed for the convergence in the finite sum case within Theorem \ref{teoASfs} and provide the worst-case complexity analysis.
We start the analysis by stating the conditions on the function under the expectation in problem \eqref{prob1}.

\begin{assumption}
Function  $\tilde{f}(\cdot,\xi)$  is continuous and convex on $\Omega$ for any given $\xi$  and there exists a solution $x_{\cal{N}}^*$ of problem $\min_{x \in \Omega} f_{\cal{N}}(x)$ for any given ${\cal{N}}$. 
\label{pp1}
\end{assumption}
 The previous assumption implies that all the sample functions $f_{\mathcal{N}_{k}}$ are also convex and continuous on $\Omega$. 
  Moreover, notice that the existence of a solution $x_{\cal{N}}^*$ of problem $\min_{x \in \Omega} f_{\cal{N}}(x)$ is guaranteed if the feasible set $\Omega$ is compact or if the objective function is convex and coercive.
We state the first main result below. 

\begin{theorem}
\label{teoNk}
Suppose that Assumption A\ref{pp1} holds and that $\Omega$ is closed and convex. Then the sequence $\{N_k\}_{k \in \mathbb{N}}$ generated by AN-SPS satisfies
\begin{equation} \label{hnk0} \lim_{k \to \infty} h(N_k)=0.\end{equation}
\end{theorem}

\begin{proof}  
First, we show that retaining the same sample pushes $\theta_k$ to zero\footnote{This part of the proof uses the elements of the analysis in \cite{nasdrugi}, but it also brings new steps and thus we provide it in a complete form.}.
Assume that   $N_k=N$ for all $k\geq \tilde{k}$ and  some $N<\infty,\;  \tilde{k} \in \mathbb{N}$.  According to  Step S5 of AN-SPS algorithm,  this means that  $\mathcal{N}_k=\mathcal{N}_{\tilde{k}}=:\mathcal{N}$ for all $k \geq \tilde{k}$.
Let us show that this implies boundedness of  $\{x_k\}_{k\in \mathbb{N}}$. Notice that for all $k$ the step size and the search direction are bounded, more precisely,  $\alpha_k\leq \bar{\alpha}_k \leq 1$ and 
$$\|p_k\|=\|\zeta_k v_{k}\|\leq \bar{\zeta} \|v_k\|\leq \bar{\zeta}.$$
Thus, the $\tilde{k}$  initial iterates must be bounded, i.e., there must exist $C_{\tilde{k}} $ such that $\|x_k\| \leq C_{\tilde{k}}$ for all $k=0,1,...,\tilde{k}.$ Now, let us observe the remaining sequence of iterates, i.e., $\{x_{\tilde{k}+j}\}_{j \in \mathbb{N}}$.
Let 
 $x_{\mathcal{N}}^*$ be an arbitrary solution of the problem $\min_{x \in \Omega} f_{\mathcal{N}}(x)$. Notice  that the convexity of $f_{\mathcal{N}}$ and the fact that  $\bar{g}_{k}\in  \partial f_{\mathcal{N}}(x_k)$ 
 for all $k \geq \tilde{k}$  imply that 
 $$-\overline{g}^T_{k} \left(x_{k}-x\right)\leq f_{\mathcal{N}}(x)-f_{\mathcal{N}}(x_k)$$ 
 for all $k \geq \tilde{k}$ and  all $x\in \mathbb{R}^n$.
Then, by using nonexpansivity of the projection operator and the fact that $x_{\mathcal{N}}^* \in \Omega$, for all $k \geq \tilde{k}$ we obtain
\begin{eqnarray}
\label{e1}
||x_{k+1}-x_{\mathcal{N}}^*||^2 &=& \nonumber ||P_{\Omega}(z_{k+1})-P_{\Omega}(x_{\mathcal{N}}^*)||^2\\ \nonumber
&\leq &\nonumber
||z_{k+1}-x_{\mathcal{N}}^*||^2 = ||x_{k}-\alpha_k \zeta_k v_k-x_{\mathcal{N}}^*||^2 \\ \nonumber
&=& 
||x_{k}-x_{\mathcal{N}}^*||^2-2\alpha_k \zeta_k \frac{1}{q_k}\overline{g}^T_{k} \left(x_{k}-x_{\mathcal{N}}^*\right) + \alpha_k^2 \zeta^2_k || v_{k}||^2 \\ \nonumber
&\leq & 
||x_{k}-x_{\mathcal{N}}^*||^2+2\alpha_k \frac{\zeta_k}{q_k} (f_{\mathcal{N}_{k}}(x_{\mathcal{N}}^*)-f_{\mathcal{N}_{k}}(x_k))+ \alpha_k^2 \bar{\zeta}^2\\ 
&\leq &
||x_{k}-x_{\mathcal{N}}^*||^2 +\alpha_k^2 \bar{\zeta}^2.
\end{eqnarray}
In the last inequality, we use the fact that $\mathcal{N}_k=\mathcal{N}$ for all $k \geq \tilde{k}$.  Thus,  
$$f_{\mathcal{N}_{k}}(x_{\mathcal{N}}^*)-f_{\mathcal{N}_{k}}(x_k)=f_{\mathcal{N}}(x_{\mathcal{N}}^*)-f_{\mathcal{N}}(x_k)\leq 0$$ 
and since 
$\alpha_k \zeta_k / q_k >0$ we obtain the result.  
Furthermore, by using the induction argument, we obtain that for every  $p \in \mathbb{N}$ there holds 
\begin{eqnarray}\nonumber
||x_{\tilde{k}+p}-x_{\mathcal{N}}^*||^2  &\leq&   ||x_{\tilde{k}}-x_{\mathcal{N}}^*||^2+\bar{\zeta}^2 \sum_{j=0}^{p-1} \alpha_{\tilde{k}+j}^2 \leq  ||x_{\tilde{k}}-x_{\mathcal{N}}^*||^2+\bar{\zeta}^2 \sum_{j=0}^{\infty} \alpha_j^2
\\\nonumber
&\leq& ||x_{\tilde{k}}-x_{\mathcal{N}}^*||^2+\bar{\zeta}^2 C_2^2 \sum_{j=0}^{\infty}\frac{1}{k^2}=:\bar{C}_{\tilde{k}}< \infty. 
\end{eqnarray}
Thus, we conclude that the sequence of iterates must be bounded, i.e., there exists a compact set $\bar{\Omega} \subseteq \Omega$ such that $\{x_k\}_{k \in \mathbb{N}} \subseteq \bar{\Omega}$. Since the function $f_{\mathcal{N}}$ is convex due to Assumption A\ref{pp1}, it follows that $f_{\mathcal{N}}$ is locally Lipschitz continuous. Moreover, it is (globally) Lipschitz continuous on the compact set $\bar{\Omega}$. Let us denote the corresponding Lipschitz constant by $L_{\bar{\Omega}}$. Then, we know that $\|g\| \leq L_{\bar{\Omega}}$ holds for any $g \in \partial f_{\mathcal{N}}(x)$ and any $x \in \bar{\Omega}$ (see for example \cite{Pol87} or \cite{Sho85}). Having in mind that $\bar{g}_{k}\in  \partial f_{\mathcal{N}}(x_k)$ for all $k \geq \tilde{k}$, we conclude that  $\|\bar{g}_{k}\| \leq L_{\bar{\Omega}}$ for all $k \geq \tilde{k}$. 

Now, we prove that 
\begin{equation} 
\label{li1}
    \liminf_{k \to \infty} f_{\mathcal{N}}(x_k)=f^*_{\mathcal{N}},
\end{equation}
 where $f^*_{\mathcal{N}}=\min_{x \in \Omega} f_{\mathcal{N}}(x)$. Suppose the contrary, i.e., there exists $\varepsilon_{\mathcal{N}}>0$ such that for all $k \geq \tilde{k}$ there holds $f_{\mathcal{N}}(x_k)-f^*_{\mathcal{N}}\geq 2 \varepsilon_{\mathcal{N}} $. Recall that Assumption A\ref{pp1} implies that $f^*_{\mathcal{N}}$ is finite and that $f_{\mathcal{N}}$ is continuous. Therefore, there exists a sequence $\{y^{\mathcal{N}}_j\}_{j \in \mathbb{N}} \in \Omega $ such that $\lim_{j \to \infty} f_{\mathcal{N}}(y^{\mathcal{N}}_j)=f^*_{\mathcal{N}}$. Moreover, there exists 
a point $\tilde{y}_ {\mathcal{N}} \in \Omega$ such that 
$$f_{\mathcal{N}} (\tilde{y}_ {\mathcal{N}})<f^*_{\mathcal{N}}+\varepsilon_{\mathcal{N}}.$$ Therefore, we conclude that for all $k \geq \tilde{k}$ there holds 
$$f_{\mathcal{N}}(x_k)\geq 
f^*_{\mathcal{N}}+ 2 \varepsilon_{\mathcal{N}}=f^*_{\mathcal{N}}+\varepsilon_{\mathcal{N}} + \varepsilon_{\mathcal{N}}> f_{\mathcal{N}} (\tilde{y}_ {\mathcal{N}})+\varepsilon_{\mathcal{N}},$$
and thus for all $k \geq \tilde{k}$ we have 
$$-\overline{g}^T_{k} \left(x_{k}-\tilde{y}_{\mathcal{N}}\right)\leq f_{\mathcal{N}}(\tilde{y}_{\mathcal{N}})-f_{\mathcal{N}}(x_k)\leq -\varepsilon_{\mathcal{N}}.$$
Following the same steps  as in \eqref{e1} and using the previous inequality, we conclude that 
for all $k \geq \tilde{k}$ there holds 
\begin{eqnarray}
\label{e2}
||x_{k+1}-\tilde{y}_{\mathcal{N}}||^2 &\leq &\nonumber ||z_{k+1}-\tilde{y}_{\mathcal{N}}||^2\\ \nonumber
&\leq &
||x_{k}-\tilde{y}_{\mathcal{N}}||^2-2\alpha_k \zeta_k \frac{1}{q_k}\overline{g}^T_{k} \left(x_{k}-\tilde{y}_{\mathcal{N}}\right) + \alpha_k^2 \zeta^2_k || v_{k}||^2 \\ \nonumber
&\leq & 
||x_{k}-\tilde{y}_{\mathcal{N}}||^2-2\alpha_k \frac{\zeta_k}{q_k} \varepsilon_{\mathcal{N}}+ \alpha_k^2 \bar{\zeta}^2\\ 
&\leq &
||x_{k}-\tilde{y}_{\mathcal{N}}||^2 -2\alpha_k \frac{1}{q_k} \underline{\zeta}  \varepsilon_{\mathcal{N}}+\alpha_k^2 \bar{\zeta}^2.
\end{eqnarray}
Now, using the fact that 
\begin{equation} 
\label{qrec} 
q_k=\max \{1,\|\bar{g}_{k}\| \}\leq \max \{1,L_{\bar{\Omega}} \}:=q,
\end{equation}
we conclude that for all $k \geq \tilde{k}$ there holds 
$$||x_{k+1}-\tilde{y}_{\mathcal{N}}||^2\leq ||x_{k}-\tilde{y}_{\mathcal{N}}||^2 -2\alpha_k \frac{1}{q} \underline{\zeta}  \varepsilon_{\mathcal{N}}+\alpha_k^2 \bar{\zeta}^2
=||x_{k}-\tilde{y}_{\mathcal{N}}||^2 -\alpha_k (\frac{2}{q} \underline{\zeta}  \varepsilon_{\mathcal{N}}-\alpha_k \bar{\zeta}^2).$$
Since $\alpha_k \leq C_2 /k$, there holds $\lim_{k \to \infty} \alpha_k=0$ and thus there must exist $\bar{k}\geq \tilde{k}$ such that $\alpha_k \bar{\zeta}^2 \leq \frac{1}{q} \underline{\zeta}  \varepsilon_{\mathcal{N}}=:\underline{\varepsilon}_{\mathcal{N}}$. Therefore, we have 
$$||x_{k+1}-\tilde{y}_{\mathcal{N}}||^2\leq ||x_{k}-\tilde{y}_{\mathcal{N}}||^2 -\alpha_k \underline{\varepsilon}_{\mathcal{N}}.$$ 
Moreover, for any $p \in \mathbb{N}$ there holds 
$$||x_{\bar{k}+p}-\tilde{y}_{\mathcal{N}}||^2\leq ||x_{\bar{k}}-\tilde{y}_{\mathcal{N}}||^2 - \underline{\varepsilon}_{\mathcal{N}} \sum_{j=0}^{p-1} \alpha_{\bar{k}+j}$$
and letting $p \to \infty$ we obtain the contradiction since $\sum_{k=0}^{\infty} \alpha_k \geq \sum_{k=0}^{\infty} 1/k=\infty. $ 
Thus, we conclude that \eqref{li1} must hold. Therefore there exists  $K_1 \subseteq \mathbb{N}$ such that $$\lim_{k \in K_1} f_{\mathcal{N}}(x_k)=f^*_{\mathcal{N}}$$ and since the sequence of iterates is bounded, there exists $K_2 \subseteq K_1$  and a solution $\tilde{x}_{\mathcal{N}}^*$ of the problem  $\min_{x \in \Omega} f_{\mathcal{N}}(x)$ such that 
\begin{equation}
\lim_{k \in K_2} x_k=\tilde{x}_{\mathcal{N}}^*.
\label{novo314}
\end{equation}
Now, we show that the whole sequence of iterates converges. 
Let 
\begin{equation} 
\label{rec1prvo} 
\lbrace{ x_k \rbrace}_{k\in K_2}:=\lbrace{ x_{k_i} \rbrace}_{i\in\mathbb{N}}.
\end{equation} 
 Following the steps of  \eqref{e1} we obtain that the following holds for any $s\in\mathbb{N}$
$$||x_{k_{i}+s}-\tilde{x}_{\mathcal{N}}^*||^2 
\leq 
||x_{k_{i}}-\tilde{x}_{\mathcal{N}}^*||^2 + \bar{\zeta}^2 \sum_{j=0}^{s-1} \alpha_{k_{i}+j}^2\leq ||x_{k_{i}}-\tilde{x}_{\mathcal{N}}^*||^2 + \bar{\zeta}^2 \sum_{j=k_{i}}^{\infty} \alpha_{j}^2=:a_i.$$
Thus, 
for any  $s, m \in \mathbb{N} $ there holds
$$||x_{k_i+s}-x_{k_i+m}||^2 \leq 2 ||x_{k_i+s}-\tilde{x}_{\mathcal{N}}^*||^2+ 2 ||x_{k_i+m}-\tilde{x}_{\mathcal{N}}^*||^2 \leq 4 a_i.$$
Since  $\sum_{j=k_i}^{\infty}\alpha_j^2$ is a residual of convergent sum and  \eqref{novo314} holds,   we have  $$\lim_{i\to\infty}a_i=0.$$ 
Therefore, for every $\varepsilon > 0$ there exists $ k_i \in \mathbb{N}$ such that for all $t,l \geq k_i $  there holds $||x_{t}-x_{l}|| \leq \varepsilon$, i.e., the sequence $\{x_k\}_{k \in \mathbb{N}}$ is a Cauchy sequence and thus convergent. This, together with \eqref{novo314}, implies  $$\lim_{k\to\infty}x_k=\tilde{x}_{\mathcal{N}}^*,$$ and Step S3 of AN-SPS algorithm implies 
$$\lim_{k \to \infty} \theta_k=\lim_{k \to \infty} \|s_k\|=\lim_{k \to \infty} \|x_{k+1}-x_k\|=0. $$
This completes the first part of the proof, i.e., we have just proved that if the sample is kept fixed, the sequence $\{\theta_k\}_{k \in \mathbb{N}}$ tends to zero. 

Finally, we prove the main result \eqref{hnk0}. 
Assume the contrary. Since the sequence $\{h(N_k)\}_{k\in \mathbb{N}}$ is nonincreasing, this means that we assume the existence of     $\bar{h}>0$  such that $h(N_k)\geq  \bar{h}$ for all $k\in \mathbb{N}$. This further implies that there exists  $N<N_{\infty}$ and $\bar{k} \in \mathbb{N}$ such that $N_k=N$ for all $k \geq \bar{k}$, where $N_{\infty}=\infty$ in unbounded sample case and $N_{\infty}$ coincides with the full sample size in the bounded sample (finite sum) case. Thus, according to the Step S5 of AN-SPS algorithm, there holds that $$\theta_k  \geq h(N_k)=h(N) \geq \bar{h}>0$$ for all $k \geq \bar{k}$, since we would have an increase of the sample size $N$ otherwise. 
On the other hand, we have just proved that if the sample size is fixed, then  
$$\lim_{k \to \infty} \theta_k=0,$$
which is in contradiction with $\theta_k  \geq \bar{h}>0$.  Thus, we conclude that   $$\lim_{k \to \infty} h(N_k)=0,$$ which completes the proof.
\end{proof}

Next, we analyze the conditions that provide a sequence of bounded iterates generated by  AN-SPS algorithm. Let us define the SAA error sequence as follows, \cite{nasdrugi},  
\begin{equation}
\label{ebar}
\bar{e}_k=|f_{\mathcal{N}_k}(x_k)-f(x_k)|+|f_{\mathcal{N}_k}(x^*)-f(x^*)|,
\end{equation} 
where $x^*$ is an arbitrary solution of \eqref{prob1}. 
The proof of the following proposition is similar to the proof of Proposition 3.4 of \cite{nasdrugi}, but the conditions are relaxed since we have $N_k \to \infty$ as a consequence of the Theorem \ref{teoNk}. Moreover,  scaling of the subgradients relaxes the assumption of uniformly bounded $\bar{g}_k$ sequence. Although the modifications are mainly technical, we provide the proof in the Appendix (Section \ref{sec6}) for the sake of completeness.  Condition \eqref{sumek} in the sequel states the sample size growth under which we achieve bounded iterates. For instance, in the cumulative sample case, $N_k=k$ is sufficient to ensure this condition. Although we believe that the condition is not too strong, it is still an assumption and not the consequence of the algorithm, so this issue remains an open question for future work. 

\begin{proposition} \label{teoBounded}
Suppose that $\Omega$ is closed and convex,  Assumption A\ref{pp1} holds and $\lbrace x_k \rbrace_{k \in \mathbb{N}}$ is a sequence generated by Algorithm AN-SPS. Then there exists a compact set $\bar{\Omega}\subseteq\Omega$ such that $\lbrace x_k \rbrace \subseteq \bar{\Omega}$ provided that 
\begin{equation} \label{sumek}
    \sum_{k=0}^{\infty} \bar{e}_k/k \leq C_4 <\infty,
\end{equation}
 where  $C_4$ is a positive constant.
\end{proposition}

As it can be seen from the proof,  $\bar{\Omega}$ stated in the previous proposition depends only on $x_0$ and given constants, so it can be (theoretically) determined independently of the sample path.     However, since we consider unbounded samples in general, we need the following assumption. 

\begin{assumption}
For every $x\in \Omega$ there exists a constant $L_x$ such that  $\tilde{f} (x,\xi)$ is locally $L_x$-Lipschitz continuous for any $\xi$. 
\label{pp_Lip}
\end{assumption}
This assumption implies that each SAA function is locally Lipschitz continuous with a constant that depends only on a point $x$ and not on a random vector $\xi$. In a bounded sample case this is obviously satisfied under assumption A\ref{pp1}, while in general, it holds for a certain class of functions - when $\xi$ is separable from $x$ for instance.   
Next, we prove the almost sure convergence of AN-SPS algorithm under the stated assumptions. Notice that \eqref{sumek} does not necessarily imply that $\lim_{k \to \infty } \bar{e}_k =0$. Thus, we add a common assumption in stochastic analysis in order to ensure a.s. convergence of the sequence $\{\bar{e}_k\}_{k \in \mathbb{N}}$. 

\begin{assumption}
The function $\tilde{f}$ is dominated by a P-integrable function on any compact subset of $\mathbb{R}^n$.
\label{pp_dominated}
\end{assumption}

Under the stated assumptions, the  Uniform Law of Large Numbers (ULLN) implies (Theorem 7.48 in \cite{shapiro})
\begin{equation}
 \lim_{N \rightarrow \infty}\sup_{x \in S}|f_{N}(x)-f(x)|=0 \quad \mbox{a.s.}   
 \label{Shapiro-ULLN}
 \end{equation}  for  any compact subset  $S \subseteq \mathbb{R}^n$. This will further imply the a.s. convergence of the sequence $\lbrace \bar{e}_k \rbrace_{k \in \mathbb{N}}$. Notice that $\lim_{k \to \infty } \bar{e}_k =0$ is satisfied in the bounded sample case, as well as \eqref{sumek} since AN-SPS achieves the full sample eventually. In that case, the assumptions A\ref{pp_Lip} and A\ref{pp_dominated} are not needed for the convergence result. 
 
 {\bf{Remark}:} The following theorem states a.s. convergence of the proposed method. Although it follows the same steps, the proof differs from the proof of Theorem 3.1 of \cite{nasdrugi} in several places. Under the stated assumptions, we prove that the sample size tends to infinity and that the iterates remain within a compact set. After that, the proof follows the steps of the proof in \cite{nasdrugi} completely, except for the scaling of the subgradient in Step S1 of AN-SPS algorithm. This alters the inequalities, but the Assumption A\ref{pp_Lip} implies that $q_k$ can be uniformly bounded from above and below, thus the main flow remains the same. We state the proof in the Appendix (Section \ref{sec6}) for completeness.

\begin{theorem}
\label{teoAS}
Suppose that Assumptions A\ref{pp1}-A\ref{pp_dominated} and \eqref{sumek} hold and that $\Omega$ is  closed and convex. Then the sequence $\{x_k\}_{k \in \mathbb{N}}$ generated by AN-SPS converges to a solution of problem \eqref{prob1}  almost surely.  
\end{theorem}

 Finally, we state the results for the finite sum problem as an important class of \eqref{prob1} 
\begin{equation}
    \label{fsprob1} 
    \min_{x \in \Omega}  \frac{1}{N} \sum_{i=1}^{N} f_i(x).
\end{equation}
As we mentioned before, Assumption A\ref{pp_dominated} is redundant in this case as well as \eqref{sumek} since $\bar{e}_k=0$ for all $k$ large enough. Moreover, Assumption A\ref{pp_Lip} is also satisfied due to the fact that there are only finitely many functions $f_i$. At the end, notice that under Assumption A\ref{pp1}, the full sample is eventually achieved and the proof of Theorem \ref{teoNk} also reveals that the convergence is deterministic. We summarise this in the next theorem. 
\begin{theorem}
\label{teoASfs}
Suppose that Assumption A\ref{pp1} holds and that $\Omega$ is closed and convex. Then the sequence $\{x_k\}_{k \in \mathbb{N}}$ generated by AN-SPS converges to a solution of problem \eqref{fsprob1}.
\end{theorem}

 We also provide the worst-case complexity analysis for the relevant finite sum problem \eqref{fsprob1}. 

\begin{theorem} \label{teocompl}
Suppose that the assumptions of Theorem \ref{teoASfs} hold and that the sample size increases as in  \eqref{increase}. Then, $\varepsilon$-vicinity of an optimal value  $f^*$ of problem \eqref{fsprob1} is reached after at most $$\hat{k}=2 \bar{k}+ \left(\frac{q (\bar{c}_1+||x_{\bar{k}}-x^*||^2)}{\underline{\zeta}}\right)^{\frac{1}{1-\delta}} \varepsilon^{\frac{1}{\delta-1}}$$
iterations, where  $$\bar{k}:=(\lceil C_2 \bar{\zeta} N\rceil+1) \frac{\log(N/N_0)}{\log(r)}, \quad \bar{c}_1:=\sum_{k=0}^{\infty} \frac{C_2^2 \bar{\zeta}^2}{k^2},$$ provided that $\alpha_k \geq k^{-\delta},$ $\delta \in [0,1)$ for all $k \in \{\bar{k}, \bar{k}+1,...,\hat{k}\}$.
\end{theorem}
\begin{proof}
Let us denote by $N^1<N^2<...<N^d$ all the sample sizes that are used during the optimization process. Then, we have that $N^1=N_0$, where $N_0$ is the initial sample size, and $N^d=N$ since we have proved that the full sample is reached eventually. Furthermore, according to \eqref{increase}, we know that $N^d \geq r^{d-1} N_0$ and thus we conclude that 
\begin{equation*} 
\label{spoljasnje} 
d-1\leq \frac{\log(N/N_0)}{\log(r)}.
\end{equation*}
Furthermore, notice  that for any $k \in \mathbb{N}$ there holds 
$$\theta_k = \|x_{k+1}-x_k\|=\|P_{\Omega}(z_{k+1})-P_{\Omega}(x_k)\|\leq \|z_{k+1}-x_k\|=\|\alpha_k p_k\|\leq \frac{C_2}{k} \bar{\zeta}.$$
Suppose that we are at iteration $k$ with a sample size $N_k=N^j$, with $j<d$. Then, according to Step S5 of Algorithm 1, the sample size $N^j$ is changed after at most 
\begin{equation*}
\label{unutrasnje} 
    \lceil \frac{C_2 \bar{\zeta}}{h(N^j)}\rceil +1
\end{equation*} 
iterations.  Moreover, since $N^j\leq N-1$ for all $j=1,...,d-1$, there must hold that $$h(N^j) \geq h(N-1)=\frac{N-(N-1)}{N}=\frac{1}{N}$$ for all $j=1,...,d-1$ and thus the number of iterations with the same sample size smaller than $N$ is uniformly bounded by $\lceil C_2 \bar{\zeta} N\rceil+1$. Thus, we conclude that after at most 
\begin{equation*}
\label{kbar} 
   \bar{k}:=(\lceil C_2 \bar{\zeta} N\rceil+1) \frac{\log(N/N_0)}{\log(r)}
\end{equation*} 
iterations the full sample size is reached.

Now, let us observe the iterations $k \geq \bar{k}$ and denote the objective function of problem \eqref{fsprob1} by  $f$. Theorem \ref{teoASfs} implies that $\lim_{k \to \infty} f(x_k)=f^*$ and thus there exists a finite iteration $k$ such that $f(x_k)< f^*+\varepsilon.$ Let us denote by $\hat{j}$ the smallest  $j \in \mathbb{N}_{0}$ such that $f(x_{\bar{k}+\hat{j}})< f(x^*)+\varepsilon$, where   $x^*$ is a solution of problem \eqref{fsprob1}. Using the same arguments as in  \eqref{e1}, we obtain 
\begin{equation}
    \label{rec1} 
    ||x_{\bar{k}+\hat{j}}-x^*||^2 \leq ||x_{\bar{k}}-x^*||^2 -\sum_{j=0}^{\hat{j}-1} 2 \alpha_{\bar{k}+j} \zeta_{\bar{k}+j}\frac{1}{q_{\bar{k}+j}} (f(x_{\bar{k}+j})-f(x^*))+
    \sum_{j=0}^{\hat{j}-1} \alpha^2_{\bar{k}+j} \zeta^2_{\bar{k}+j}.
\end{equation}
Notice that 
\begin{equation} \label{rec2a}
    \sum_{j=0}^{\hat{j}-1} \alpha^2_{\bar{k}+j} \zeta^2_{\bar{k}+j}\leq \sum_{k=0}^{\infty} \frac{C_2^2 \bar{\zeta}^2}{k^2}:=\bar{c}_1 <\infty.
\end{equation}
Moreover, using \eqref{qrec}, \eqref{rec2a}, $\zeta_k \geq \underline{\zeta}$ for all $k$,  and $$\alpha_{\bar{k}+j}\geq \frac{1}{(\bar{k}+j)^{\delta}}\geq \frac{1}{(\bar{k}+\hat{j})^{\delta}},  
 \quad f(x_{\bar{k}+j})-f(x^*)\geq \varepsilon, \quad j=0,...,\hat{j}-1,$$
 from \eqref{rec1} we obtain 
 \begin{equation*}
     \label{rec3}
     0\leq  ||x_{\bar{k}}-x^*||^2 -   \frac{2 \hat{j} \underline{\zeta} \varepsilon }{q (\bar{k}+\hat{j})^{\delta}}  +
    \bar{c}_1.
 \end{equation*}
 Finally, let us observe two cases: 1) $\hat{j} \leq \bar{k}$, and 2) $\hat{j} > \bar{k}$. In the first case, the upper bound on $\hat{j}$ is obvious. In the second case, we have 
 \begin{equation*}
     \label{rec4}
     0\leq  ||x_{\bar{k}}-x^*||^2 -   \frac{2 \hat{j} \underline{\zeta} \varepsilon }{q \hat{j}^{\delta} 2^{\delta}}  +
    \bar{c}_1\leq  ||x_{\bar{k}}-x^*||^2 -   \frac{ \hat{j}^{1-\delta} \underline{\zeta} \varepsilon }{q }  +
    \bar{c}_1,
 \end{equation*}
 and thus 
 $$\hat{j} \leq \left(\frac{q (\bar{c}_1+||x_{\bar{k}}-x^*||^2)}{\underline{\zeta}}\right)^{\frac{1}{1-\delta}} \varepsilon^{\frac{1}{\delta-1}}=:\bar{c}_2.$$
 Combining both cases we conclude that $$\hat{j} \leq \max \{\bar{c}_2,\bar{k}\}\leq \bar{c}_2+\bar{k} $$
 and thus 
 $\hat{k} \leq \bar{k}+\bar{c}_2+\bar{k}=2 \bar{k}+\bar{c}_2,$ which completes the proof. 
\end{proof}
A few words are due to this result. The number of iterations $\hat{k} $ to reach the $\varepsilon$-vicinity of the optimal value represents the worst-case complexity and it is obtained by using very conservative bounds. In this setup, the parameter $r$, which controls the increase of the sample size, influences the number of iterations to reach the full sample size through $\log(r)$. Higher $r$ yields smaller $\hat{k}$, but it also brings potentially higher computational costs as larger samples are needed to compute the approximate functions and the corresponding subgradients. Notice that the proposed algorithm requires only one subgradient per iteration, while the costs of evaluating the approximate objective function depend on the line search. However, the per-iteration costs can be controlled by the parameter $m$  which represents the maximal number of trial points at which the approximate function is evaluated during the line search. In our experiments, we set $m=2$, and we believe that this number should be modest to avoid unnecessarily detailed line search. However, choosing an optimal value for $m$, or even adaptive $m_k$, could be an interesting topic for some future research since it influences the computational cost complexity. 

The assumption $\alpha_k \geq k^{-\delta}, \delta \in [0,1)$ for all $k \in \{\bar{k}, \bar{k}+1,...,\hat{k}\}$ actually indicates that the line search condition \eqref{Armijo} is satisfied in a finite number of iterations $k \in \{\bar{k}, \bar{k}+1,...,\hat{k}\}$. Since the step sizes are upper bounded by $C_2/k$, this is possible only if we assume that $C_2$ is large enough. Notice that the acceptance of a trial point can be controlled by $F_k$. For instance, if $F_k$ is set to $f_{{\cal{N}}_k}(x_k)+C/2^k$, choosing large $C$ increases the chances of successful line search and even of accepting the full step in finitely many iterations. If this is the case, more precisely, if $\alpha_k \geq k^{-\delta}$ with $\delta=0$ for all $k \in \{\bar{k}, \bar{k}+1,...,\hat{k}\}$, we achieve the complexity of order $\varepsilon^{-1}$. 

We end this section by noticing that the complexity result with respect to the expected objective function's value, as the one in Theorem 3.4 of \cite{nasprvi}, can be achieved, but under additional sampling assumptions. Although this type of result can be helpful, we believe that the advantage of the proposed AN-SPS method lies in its ability to embed various sampling and nonmonotone line search strategies, allowing the method to adapt to the problem at hand and produce good practical behavior.  

\section{Numerical results}\label{sec4}

Within this section, we test the performance of AN-SPS algorithm on well-known binary classification data sets listed in Table \ref{tabela_dataset}. 

The problem that we consider is a constrained finite sum problem with $L_2$-regularized hinge loss local cost  functions, i.e., 
\begin{align*}
\min_{x \in \Omega} f_N (x) &: = \delta ||x||^2 + \frac{1}{N}\sum_{i=1}^{N} \max\lbrace{0, 1 - z_i x^T w_i\rbrace},\\
\Omega& : = \lbrace{x \in \mathbb{R}^n \;: ||x||^2 \leq \frac{1}{\delta} \rbrace},
    \end{align*}
where $\delta=10$ is the regularization parameter, $w_i \in \mathbb{R}^n $ are the attributes and $z_i \in \lbrace{1,-1\rbrace}  $ are the corresponding labels. 
\begin{table}[!ht]
\begin{center}
 \begin{tabular}{||c l c c  ||} 
 \hline
  & Data set & $N$ & $n$  \\ [0.5ex] 
 \hline\hline
 1 & SPLICE \cite{SPLiADL}& 3175& 60  \\ 
 \hline
 2 & MUSHROOMS \cite{MUSH}& 8124& 112  \\ 
 \hline
 3 & ADULT9 \cite{SPLiADL}& 32561& 123  \\
 \hline
 4 & MNIST  \cite{MNIST} &70000 & 784  \\ [0.5ex] 
 \hline
\end{tabular}
\caption{Properties of the data sets used in the experiments.}\label{tabela_dataset}
\end{center}
\end{table}

AN-SPS algorithm is implemented with the following parameters: $C_2=100, \eta =10^{-4}, m=2, N_0=\lceil{0.1 N \rceil}.$ The initial point $x_0$ is chosen randomly from $\Omega$. We use the method proposed in \cite[Algorithm 2, p.~1155]{kineski} with $B_k=I$ to find a descent direction  $-\overline{g}_k$ which is further scaled as in Step S1 of AN-SPS algorithm, i.e., $p_k=-\zeta_k \overline{g}_k/q_k$. 
The sample size is updated according to  Step S5 of AN-SPS and  \eqref{increase}. Recall that the sample size is increased only if $\theta_k < h(N_k)$.

We use cumulative samples, i.e., ${\cal{N}}_k \subseteq {\cal{N}}_{k+1}$ and thus, following the conclusions in \cite{greta}, we calculate the spectral coefficients based on $s_k=x_{k+1}-x_k$ and the subgradient difference $y_k=\tilde{g}_k-\bar{g_k}$, where $\tilde{g}_k \in \partial f_{{\cal{N}}_k}(x_{k+1})$. This choice requires additional costs with respect to the choice of $\tilde{g}_k=\bar{g}_{k+1}$, but it diminishes the influence of the noise since the difference is calculated on the same approximate function.  Furthermore, we test four different choices for the spectral coefficient (see   \cite{diSerafino} and the references therein for more details): 
\begin{itemize}
   \item Barzilai-Borwein 1 (BB1) \cite{BB}: 
   $$
     \lambda_k^{BB1}=\frac{s_k^T s_k}{s_k^T y_k};$$ 
    \item  Barzilai-Borwein 2 (BB2) \cite{BB}: 
    $$  \lambda_k^{BB2}=\frac{y_k^T s_k}{y_k^T y_k};$$ 
    \item Adaptive  Barzilai-Borwein (ABB) \cite{abb}:
    $$ \lambda_k :=\left\{ \begin{array}{ll}
\lambda_k^{BB2},& \frac{\lambda_k^{BB2}}{\lambda_k^{BB1}}<0.8,\\
 \lambda_k^{BB1},& \text{otherwise};
 \end{array} \right.
     $$
    \item Adaptive  Barzilai-Borwein  - minimum (ABBmin) \cite{abbmin}:
     $$\lambda_k :=\left\{ \begin{array}{ll}
 \min\lbrace{\lambda_j^{BB2}:j=\max\lbrace{1,k-m_a\rbrace},...,k\rbrace},& \frac{\lambda_k^{BB2}}{\lambda_k^{BB1}}<0.8,\\
 \lambda_k^{BB1},& \text{otherwise},
 \end{array} \right.$$
  where $m_a$ is a nonnegative integer set to 5 in our experiments.
\end{itemize}
For all the considered choices we take the following safeguard 
$$\zeta_k=\min \lbrace\overline{\zeta},\max\lbrace\underline{\zeta},\lambda_k\}\}, \quad  \underline{\zeta}=10^{-4}, \quad \overline{\zeta}=10^4.$$ 
Since the fixed step size such as $\alpha_k=1/k$ was already addressed in \cite{nasdrugi} where the results show that it was clearly outperformed by the line search LS-SPS method, we focus our attention on adaptive step size rules.  The value of $\tilde{\alpha}^1_k$ is chosen to be  $\tilde{\alpha}^1_k=\frac{1/k+\bar{\alpha}_k}{2}$, i.e., it is the middle point of the interval $\left[\frac{1}{k},\bar{\alpha}_k\right]$. Regarding the nonmonotone rule, we also test four choices  (see   \cite{NKNKJ2} and the references therein for more details): 
\begin{itemize}
    \item Maximum (MAX) \cite{grippo}: 
    $$F_k= \max_{i \in [\max \{1,k-5\}, k]}{f_{\mathcal{N}_{i}}(x_i)};  $$ 
    \item Convex combination  (CCA) \cite{zhang}: 
    $$ F_k= \max \{f_{\mathcal{N}_{k}}(x_k), D_k\},\; D_{k+1}=\frac{\eta_k q_k}{q_{k+1}}D_k+\frac{1}{q_{k+1}} f_{\mathcal{N}_{k+1}}(x_{k+1})$$ $$D_0=f_{\mathcal{N}_{0}}(x_0), \;
    q_{k+1}=\eta_k q_k+1,\;  q_0=1, \; 
   \eta_k=0.85;$$
   
    \item   Monotone rule (MON):  $$F_k=f_{\mathcal{N}_{k}}(x_k);$$
   
      \item Additional term (ADA) \cite{li}:  $$F_k= f_{\mathcal{N}_{k}}(x_k)+ \frac{1}{2^k}.$$
\end{itemize}

In order to find the best combination of the strategies proposed above, we track the objective function value and plot it against the FEV - the number of scalar products, which serves as a measure of computational cost. All the plots are in the log scale. 
In the first phase of the experiments, we test AN-SPS with different combinations of spectral coefficients and nonmonotone rules, on four different data sets.  The results reveal the benefits of the ADA rule in almost all cases, as it can be seen on representative graphs on MNIST data set (Figure \ref{mnist1v}).  In particular, as expected, more "nonmonotonicity" usually yielded better results when combined with the spectral directions.

\begin{figure}[htbp]
    \includegraphics[width=0.4\textwidth,angle = 270]{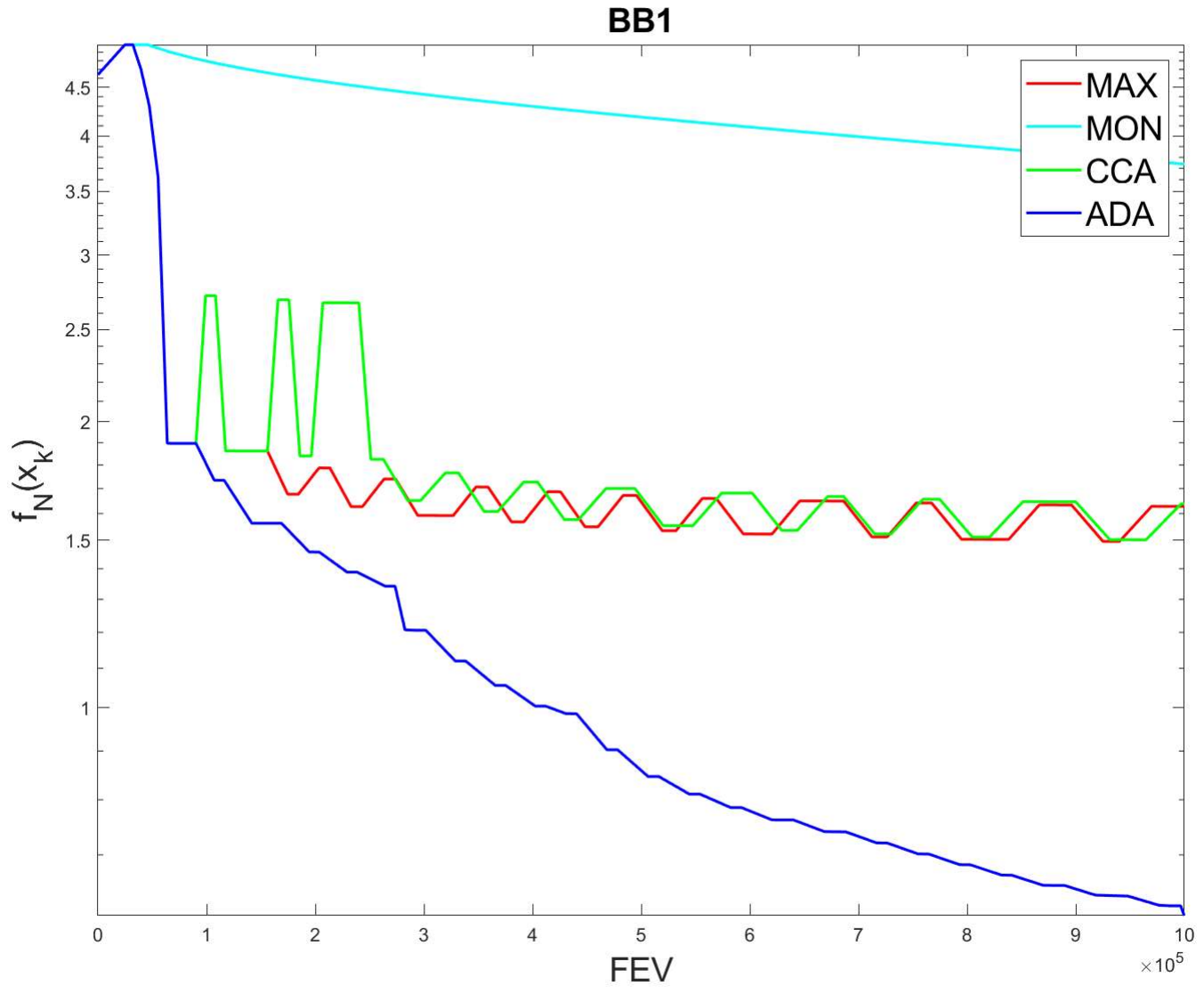}
   \includegraphics[width=0.4\textwidth,angle = 270]{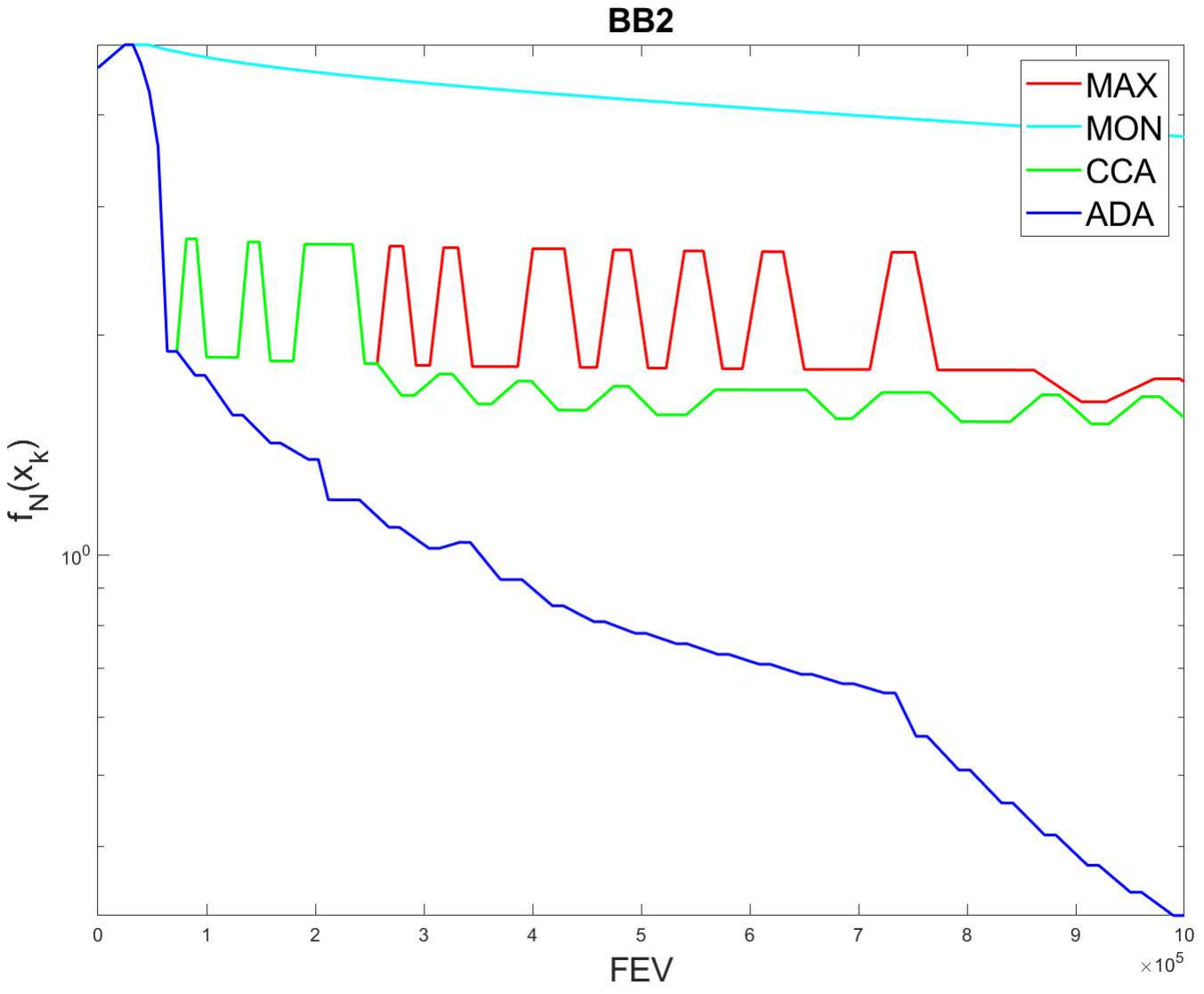}\\
       \includegraphics[width=0.4\textwidth,angle = 270]{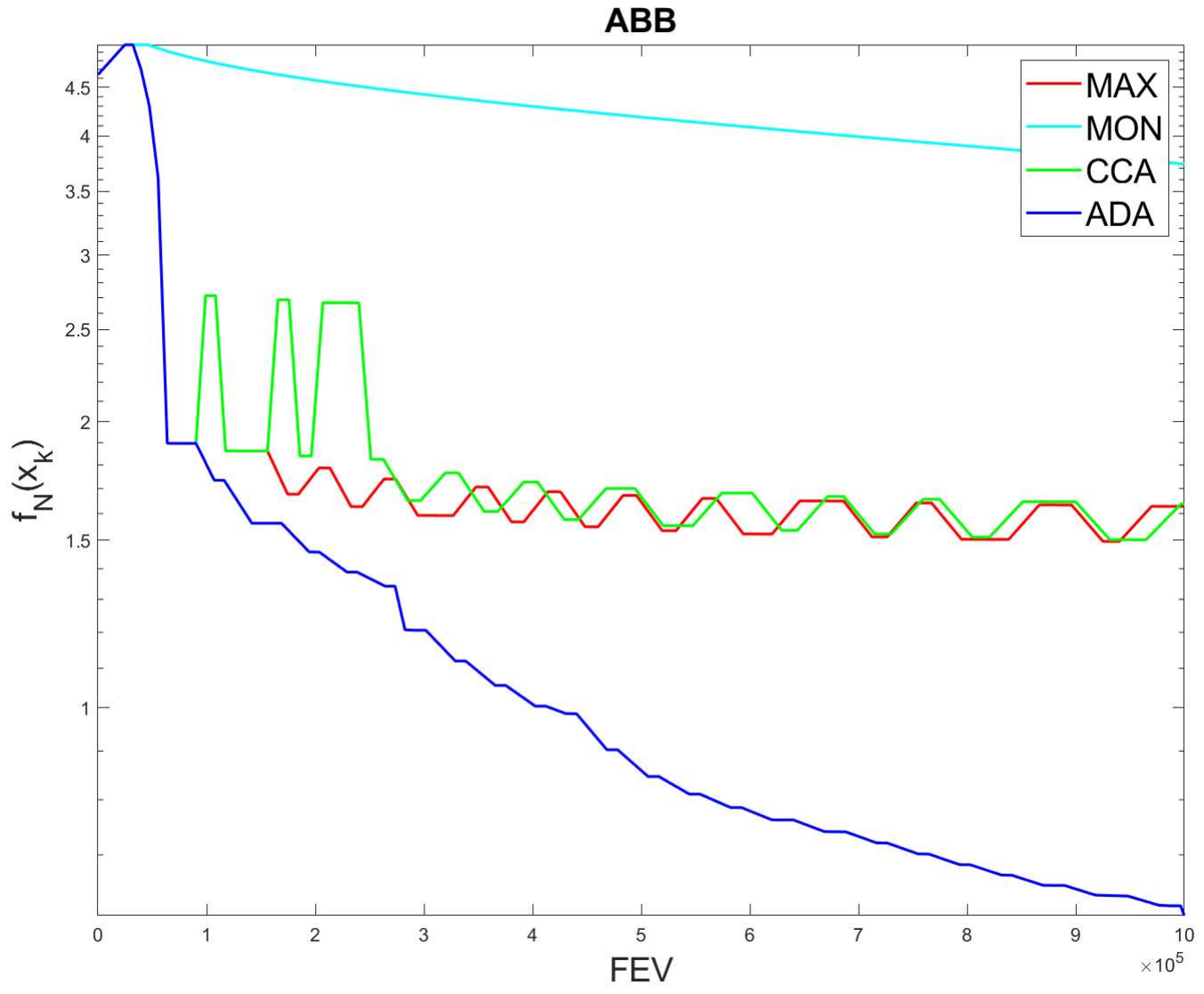}
   \includegraphics[width=0.4\textwidth,angle = 270]{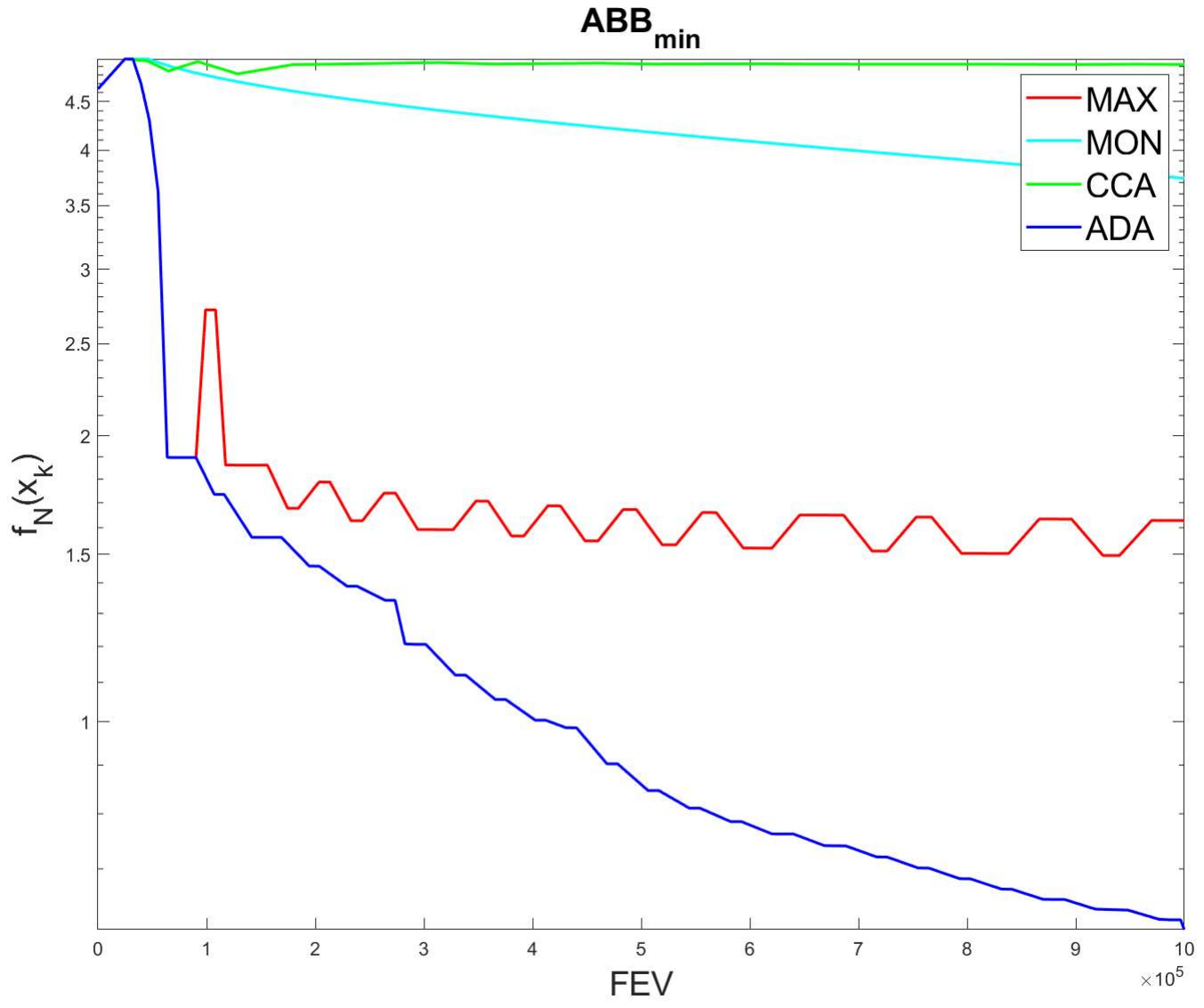}\\
    \caption{ {\footnotesize{AN-SPS algorithm with different nonmonotone rules and spectral coefficients.
    Objective function value against the computational cost (FEV). MNIST data set.}} }\label{mnist1v}
\end{figure}

Furthermore, in order to see the benefits of the adaptive sample size strategy, we compare AN-SPS with: 
\begin{itemize}
\item[1)] heuristic (HEUR) where the sample size is increased at each iteration by $N_{k+1} = \lceil\min \lbrace 1.1N_k, N\rbrace \rceil$; 
\item[2)] fixed sample  strategy (FULL) where $N_k=N$ at each iteration. 
\end{itemize}
We do the same tests for the HEUR and FULL to find the best-performing combinations of BB and line search rules. Finally, we compare the best-performing algorithms of each sample size strategy. The results for all the considered data sets are presented in Figure \ref{finale} and they show clear advantages of the adaptive sample size strategy in terms of computational costs.

\begin{figure}[htbp]
    \includegraphics[width=0.4\textwidth,angle = 270]{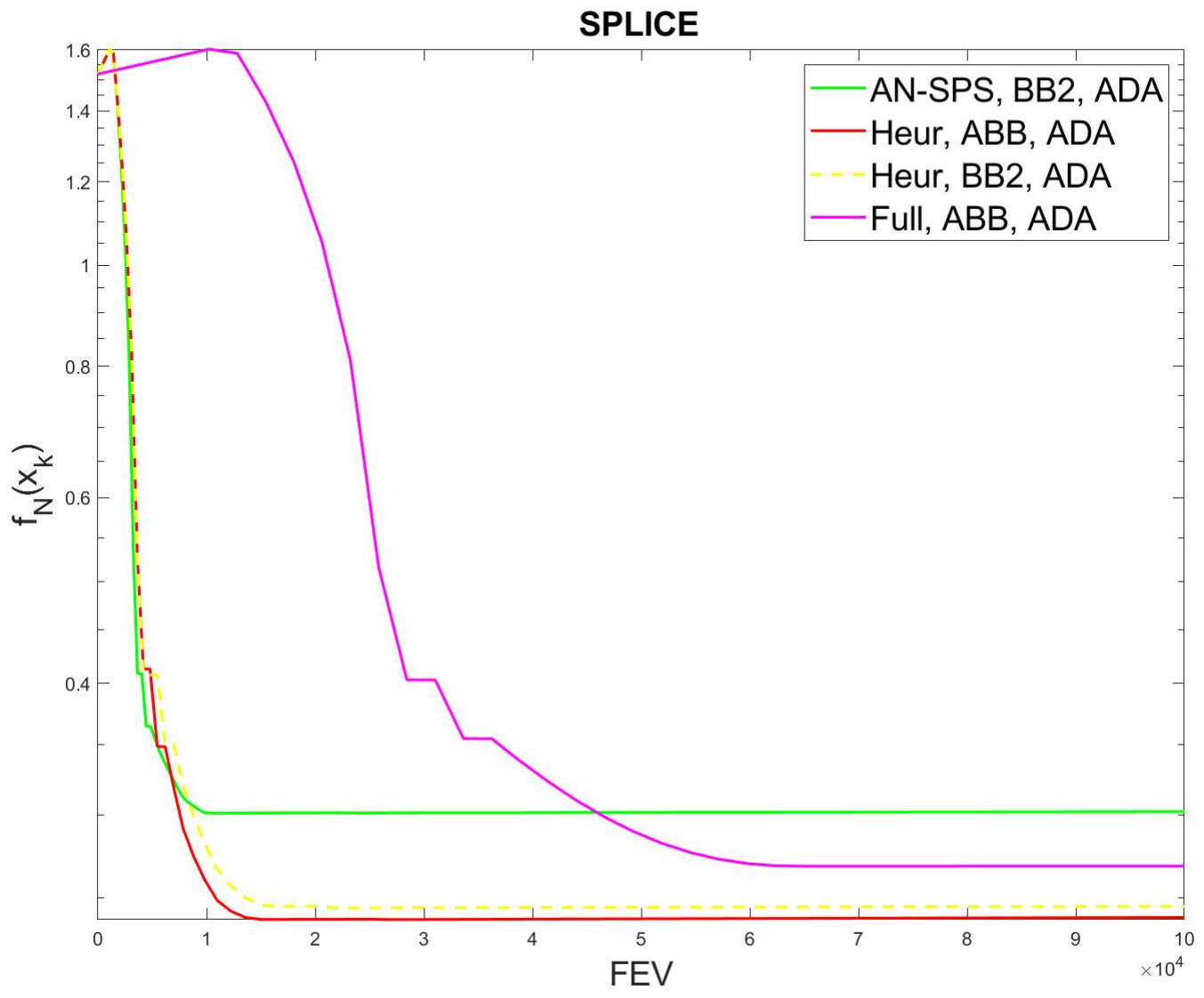}
   \includegraphics[width=0.4\textwidth,angle = 270]{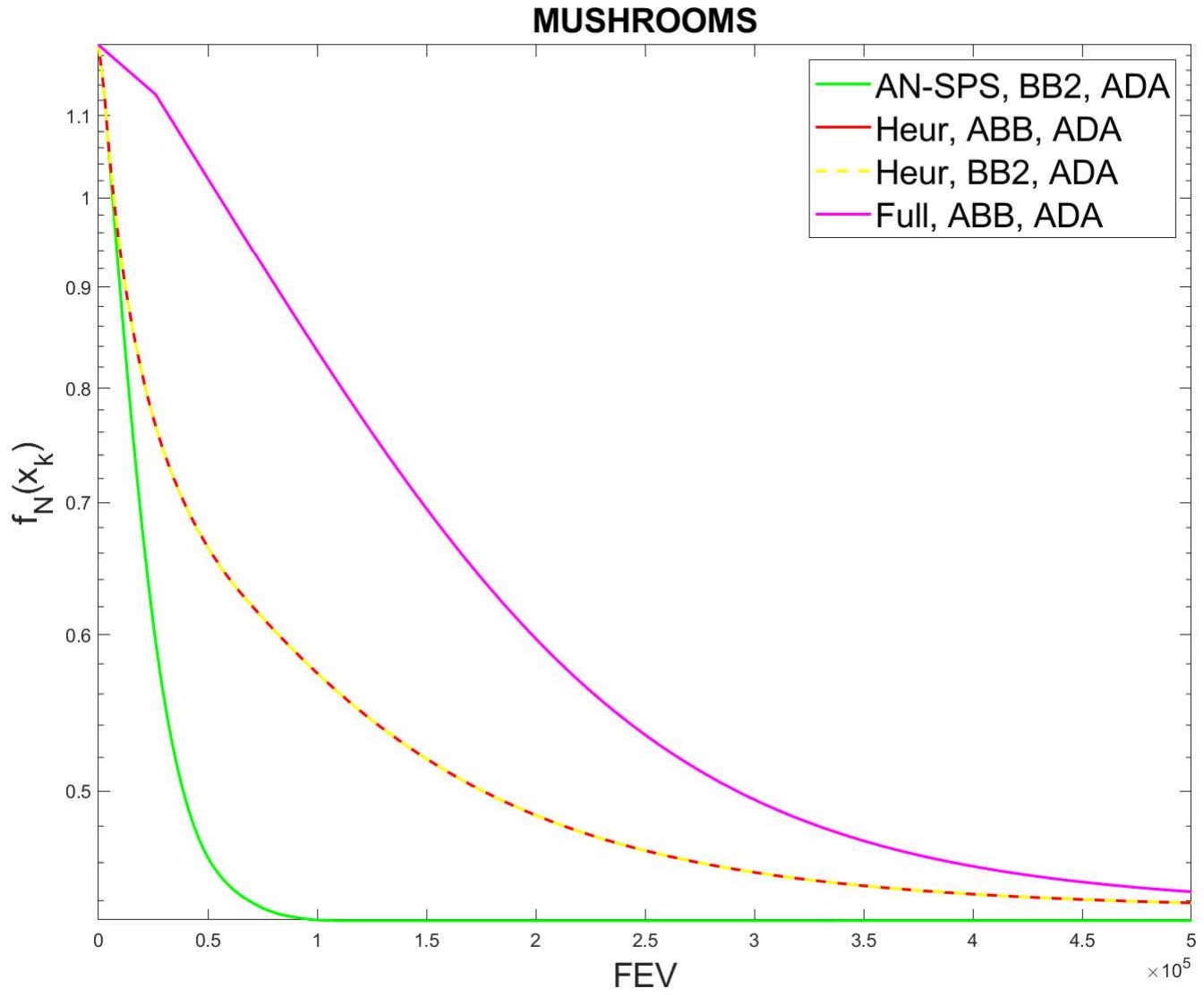}\\
       \includegraphics[width=0.4\textwidth,angle = 270]{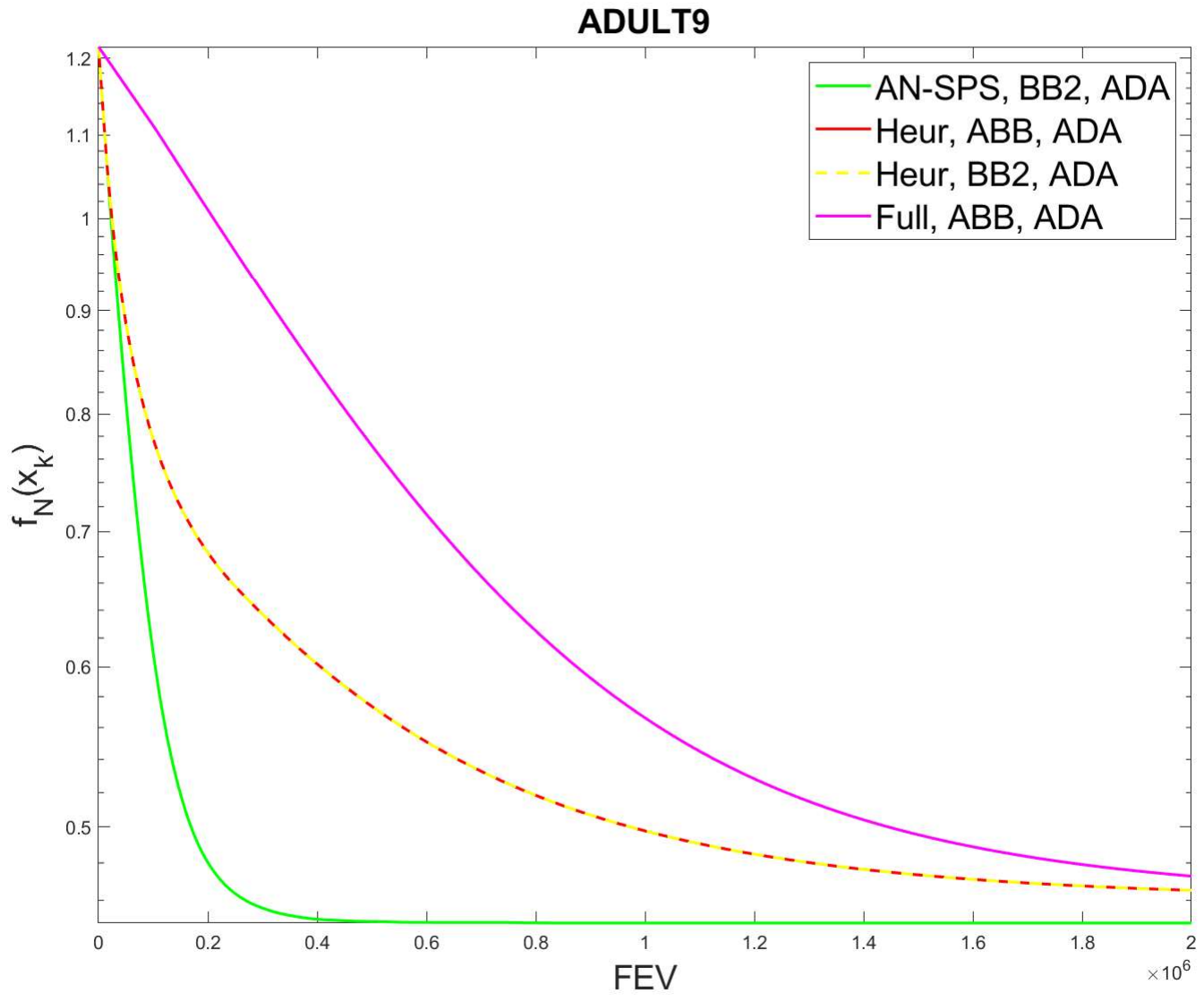}
   \includegraphics[width=0.4\textwidth,angle = 270]{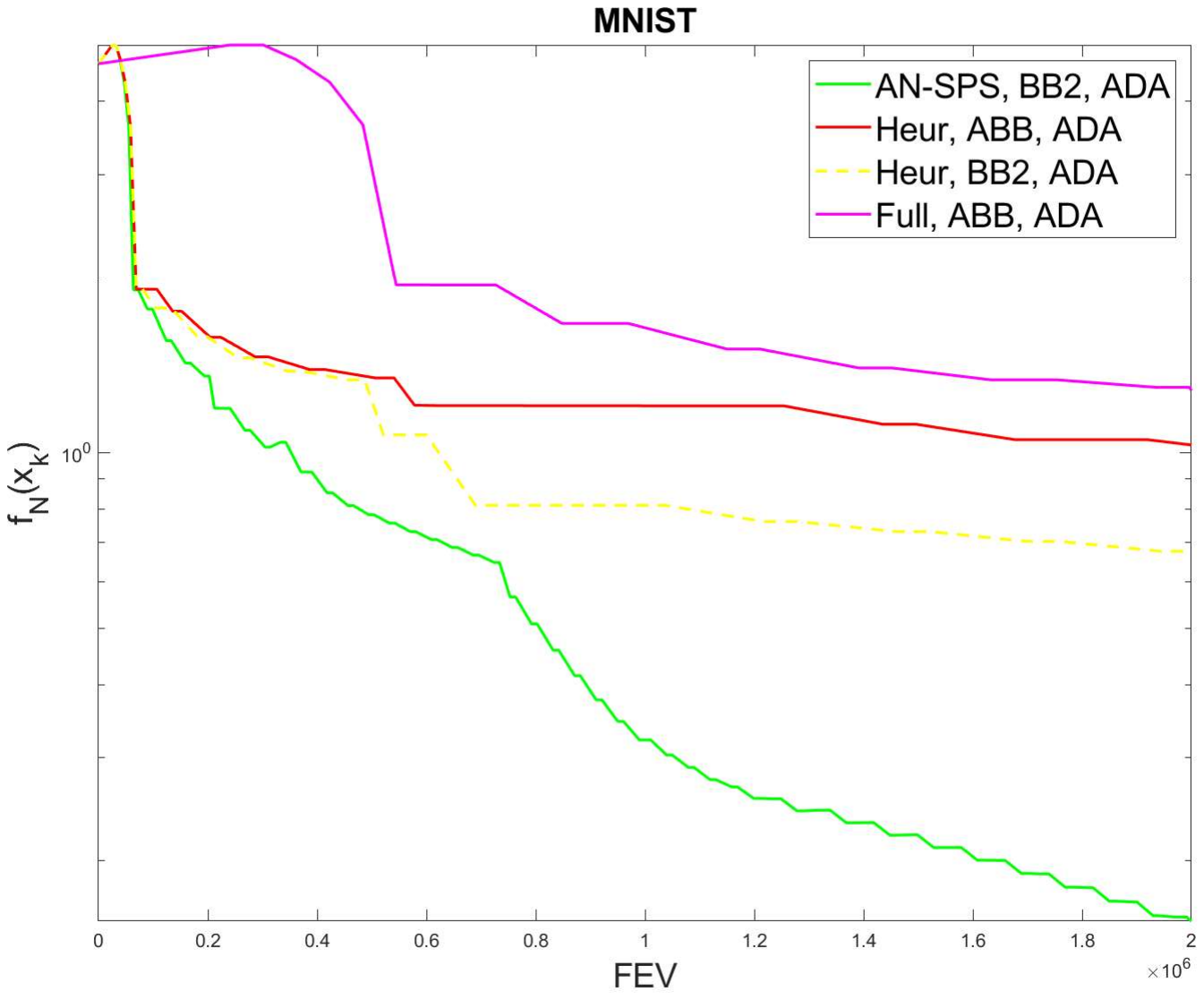}\\
    \caption{{\footnotesize{Comparison of the best-performing combinations of spectral coefficients and nonmonotone rules of AN-SPS, HEUR, and FULL sample size strategies.
        }} }\label{finale}
\end{figure}

We also provide some results on problems that are convex, but not necessarily strongly convex. In particular, we consider the same loss function, but without the $L_2$-regularization part, i.e., the following problem 
\begin{align*}
\min_{x \in \Omega} f_N (x) &: = \frac{1}{N}\sum_{i=1}^{N} \max\lbrace{0, 1 - z_i x^T w_i\rbrace}, \; 
\Omega : = \lbrace{x \in \mathbb{R}^n \;: ||x||^2 \leq 0.1 \rbrace}.
\end{align*}
Instead of going through the phase of finding the best combination for each method, we use the best-performing combinations obtained from testing the strongly convex case. The results for convex case are presented in Figure \ref{faza3_dodatno} and they show that the proposed adaptive schemes are competitive even if the regularization part is dropped.
\begin{figure}[htbp]
    \includegraphics[width=0.4\textwidth,angle = 270]{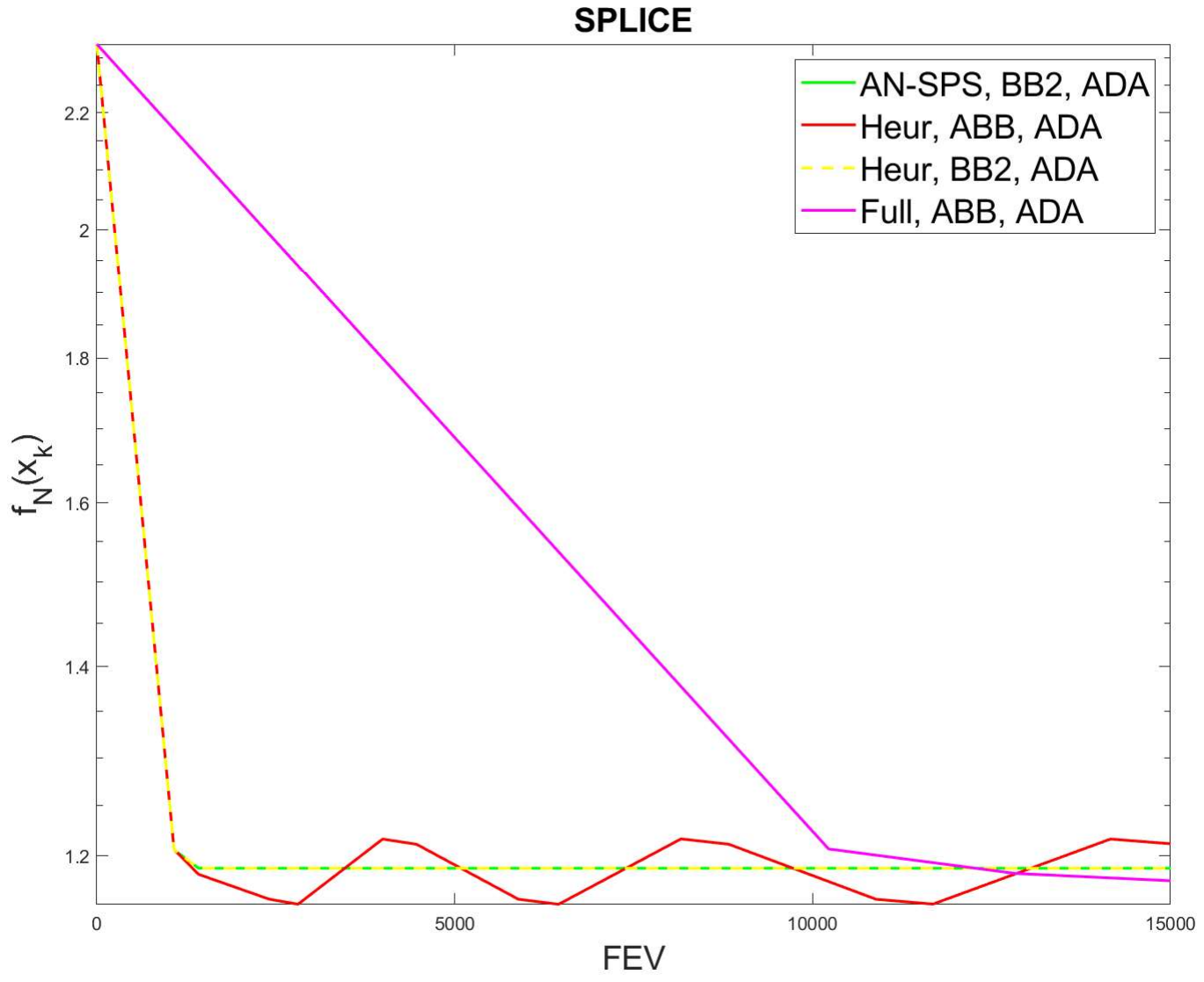}
   \includegraphics[width=0.4\textwidth,angle = 270]{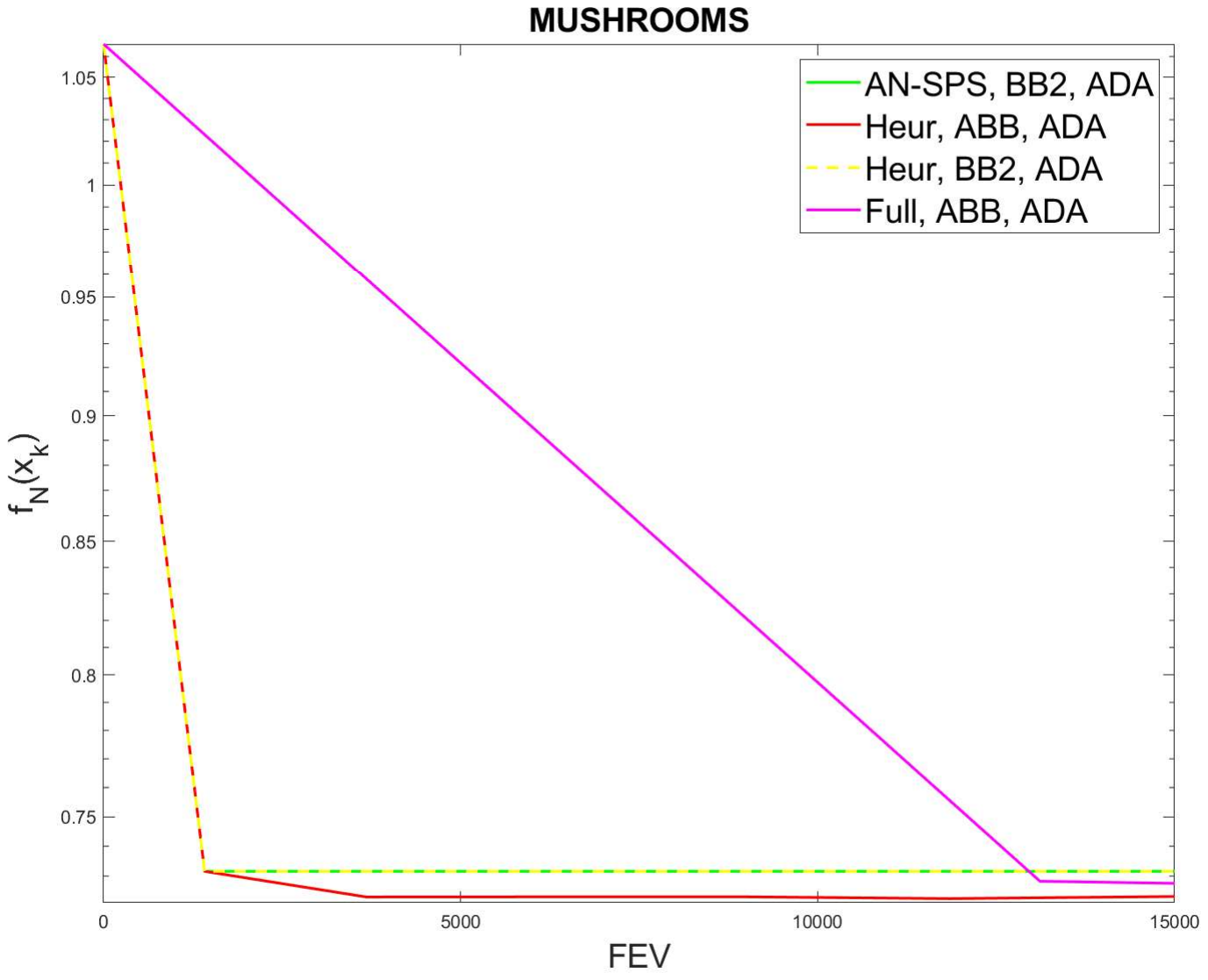}\\
       \includegraphics[width=0.4\textwidth,angle = 270]{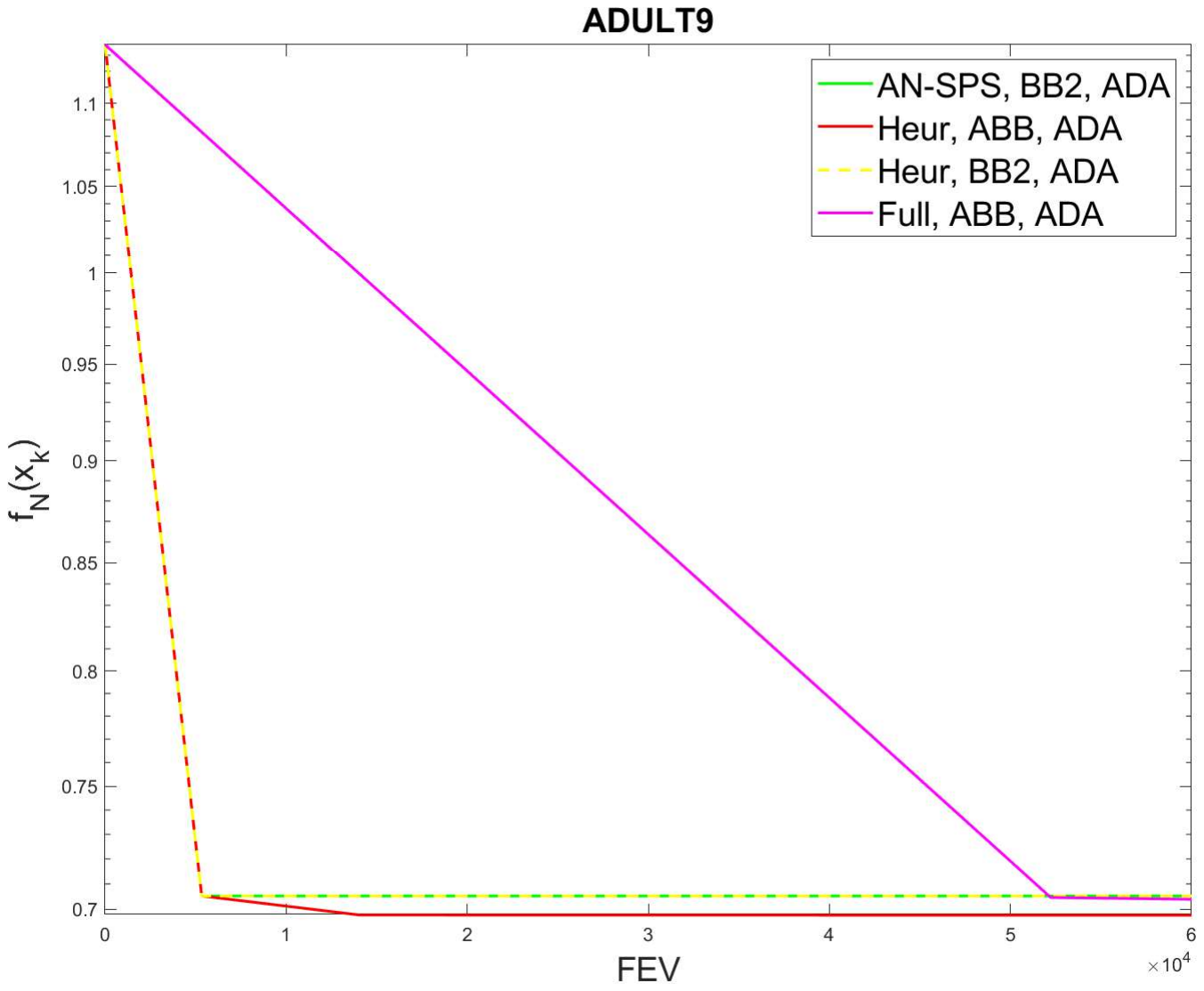}
   \includegraphics[width=0.4\textwidth,angle = 270]{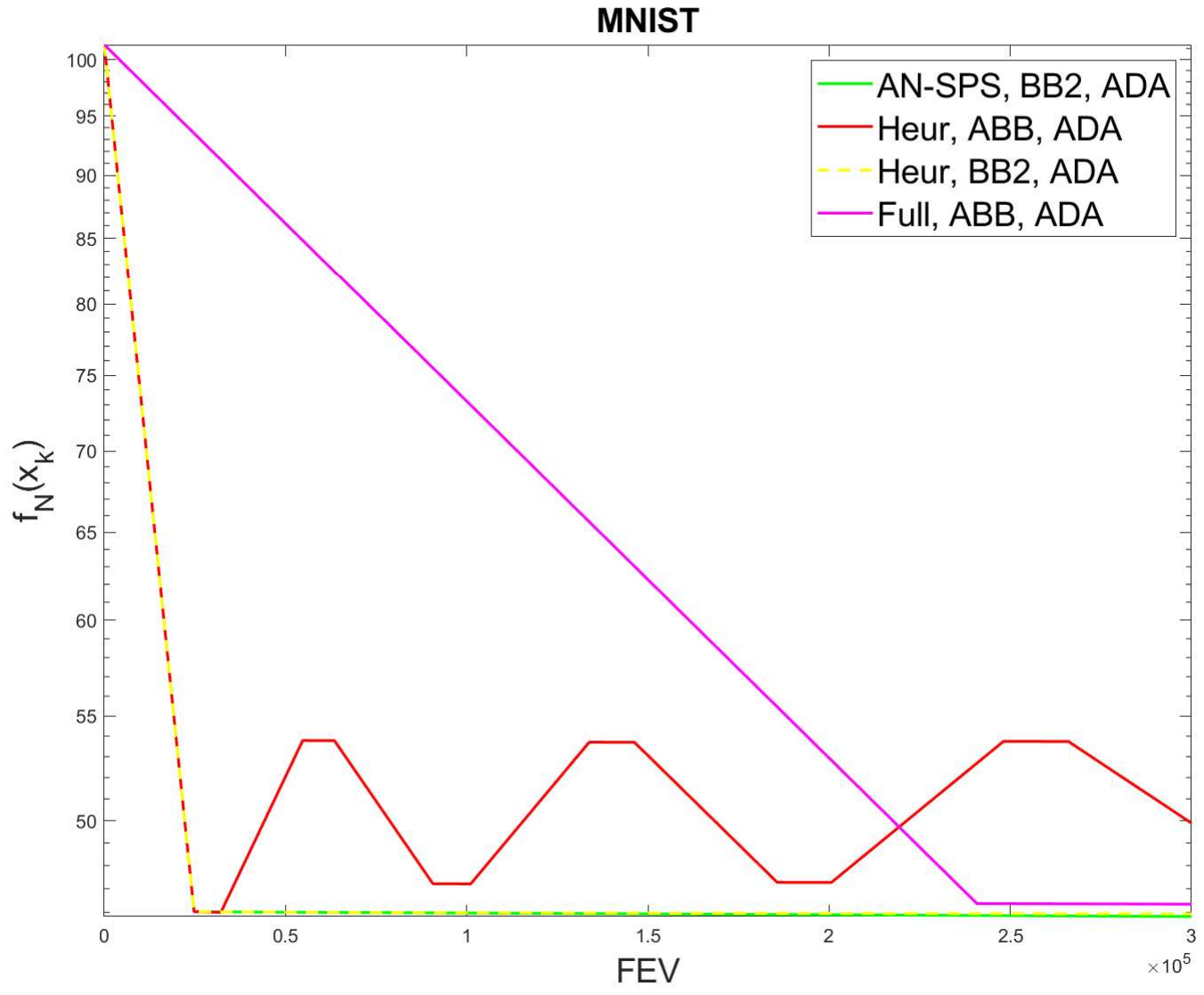}\\
    \caption{{\footnotesize{Comparison of   AN-SPS, HEUR, and FULL sample size strategies on convex problems without the $L_2$-regularization part.
        }} }\label{faza3_dodatno}
\end{figure}

\newpage
\section{Conclusions}\label{sec5}

We provide an adaptive sample size algorithm for constrained nonsmooth convex optimization problems, where the objective function is in the form of mathematical expectation, and the feasible set allows exact projections. This method allows an arbitrary (negative) subgradient direction related to the SAA function, which is further scaled and multiplied by the spectral coefficient. The coefficient can be defined in various ways and the only theoretical requirement is to keep it bounded away from zero and infinity which can be accomplished by using the standard safeguard rule. Scaling is important from a theoretical point of view since it helps us to avoid boundedness assumptions in the convergence analysis.  We proved that the method pushes the sample size to infinity and ensures that the SAA error tends to zero. On the other hand, a numerical study on Hinge loss problems showed that the adaptive strategy is efficient in terms of computational costs. Moreover, we proved that the almost sure convergence toward a solution of the original problem is attained under common assumptions in a stochastic environment. Furthermore, in the finite sum case, the convergence is deterministic and is achieved under reduced assumptions.  Moreover, we provide the worst-case complexity analysis for this case.  Since spectral coefficients are employed, we propose a nonmonotone line search over predefined intervals, although the monotone line search rule is eligible from a theoretical point of view. The numerical study also examined the performance of different line search rules and spectral coefficients. The preliminary results provide some hints for future work that may include adaptive nonmonotone strategies and inexact projections. 

{\bf{Acknowledgement.}} We are grateful to the associate editor and two anonymous referees whose comments helped us improve the paper.

\section*{Funding} The work of Nata\v sa Krklec Jerinki\' c sa has been supported by Provincial Secretariat for Higher Education and Scientific Research of Vojvodina, grant no. 142-451-2593/2021-01/2. The work of Tijana Ostoji\' c has been supported by the Ministry of Science, Technological Development and Innovation (Contract No. 451-03-65/2024-03/200156) and the Faculty of Technical Sciences, University of Novi Sad through project “Scientific and Artistic Research Work of Researchers in Teaching and Associate Positions at the Faculty of Technical Sciences, University of Novi Sad” (No. 01-3394/1).\\
{\bf{Availability statement.}} The datasets analyzed during the current study are available in the MNIST database
of handwritten digits \cite{MNIST},  LIBSVM Data: Classification (Binary Class) \cite{SPLiADL} and UCI Machine Learning Repository \cite{MUSH}.

{\bf{Disclosure statement}}\\

{\bf {Conflict of interest.}} The authors declare no competing interests.

\section{Appendix}\label{sec6}
 Recall that $X^*$ and $f^*$ are the set of solutions and the optimal value of problem  \eqref{prob1}, respectively.  

{\bf{Proof of Proposition \ref{teoBounded}.}}

\begin{proof} 
Let $x^*$ be an arbitrary solution of the problem \eqref{prob1}. Following the steps of \eqref{e1} and the definition  \eqref{ebar} we obtain for all $k=0,1,\ldots$ 
\begin{eqnarray}
\label{e12}
||x_{k+1}-x^*||^2 &=& ||P_{\Omega}(z_{k+1})-P_{\Omega}(x^*)||^2\\ \nonumber
&\leq & 
||x_{k}-x^*||^2+2\alpha_k \frac{\zeta_k}{q_k} (f_{\mathcal{N}_{k}}(x^*)-f_{\mathcal{N}_{k}}(x_k))+ \alpha_k^2 \bar{\zeta}^2\\ \nonumber
&\leq &
||x_{k}-x^*||^2+ 2\alpha_k \frac{\zeta_k}{q_k} (f(x^*)-f(x_k)+\bar{e}_k)+\alpha_k^2 \bar{\zeta}^2
\\ \nonumber
&\leq &
||x_{k}-x^*||^2+ 2\alpha_k \frac{\zeta_k}{q_k} (f(x^*)-f(x_k))+2\alpha_k \frac{\zeta_k}{q_k}\bar{e}_k+\alpha_k^2 \bar{\zeta}^2
\\ \nonumber
&\leq &
||x_{k}-x^*||^2+ 2\alpha_k \bar{\zeta}  \bar{e}_k+\alpha_k^2 \bar{\zeta}^2,
\end{eqnarray}
where we use the fact that $x_k $ is feasible and thus $f(x^*)-f(x_k)\leq 0$ and that $q_k \geq 1$. 
Further, by the induction argument and the fact that $\alpha_k \leq C_2 /k$ we obtain
$$ ||x_{k}-x^*||^2\leq ||x_{0}-x^*||^2+2 C_2 \overline{\zeta} \sum_{k=0}^{\infty} \frac{\bar{e}_k}{k}  + \overline{\zeta}^2  \sum_{k=0}^{\infty} \frac{C^2_2}{k^2}\leq C_5 <\infty .$$
This completes the proof. 
\end{proof}

{\bf{Proof of Theorem \ref{teoAS}}}

\begin{proof} First, notice that Theorem \ref{teoNk}
implies that $\lim_{k \to \infty} N_k=\infty$ in ubounded sample case.  Moreover, Proposition  \ref{teoBounded} implies that $\lbrace x_k \rbrace \subseteq \bar{\Omega}$. Furthermore, Assumption A\ref{pp_Lip} implies that for any $\mathcal{N}$ we have locally $L_{x}$-Lipschitz continuous function $f_{\mathcal{N}}(x)$. Thus, there exists a constant $L$ such that $f_{\mathcal{N}}$ is $L$-Lipschitz continuous on $\bar{\Omega}$ for any $\mathcal{N}$. This further implies that $\|\bar{g}_k\|\leq L$ for each $k$ and 
\begin{equation} \label{qkb}
    1\leq q_k \leq \max\{1, L\}:=\bar{q}. 
\end{equation}

Denote by $\mathcal{W}$ the set of all possible sample paths of AN-SPS algorithm. First we prove that  \begin{equation} \label{prvideo}
    \liminf_{k \to \infty} f(x_k)=f^* \quad \mbox{a.s.,}
\end{equation}
 where $f^*=\inf_{x \in \Omega} f(x).$ 
Suppose that $\liminf_{k \to \infty} f(x_k)=f^*$ does not happen with probability 1.
 In that case there exists a subset  of sample paths $\tilde{\mathcal{W}}\subseteq \mathcal{W}$ such that $ P(\tilde{\mathcal{W}})>0$ and 
for every $w \in \tilde{\mathcal{W}} $ there holds 
$$\liminf_{k\rightarrow \infty} f(x_k(w))> f^*,$$ 
i.e., there exists $\varepsilon(w) >0$ small enough  such that  
$f(x_k(w))-f^* \geq 2\varepsilon(w) $ for all $k$. 
   Since $f$ is assumed to be continuous and bounded from below on $\Omega$, $f^*$ is finite and we conclude that there exists a point $\tilde{y}(w)\in \Omega$ such that
$f(\tilde{y}(w))<f^*+\varepsilon(w).$
This further implies 
$$f(x_k(w))-f(\tilde{y}(w))> f(x_k(w))-f^*-\varepsilon(w) \geq 2\varepsilon(w)-\varepsilon(w)=\varepsilon(w).$$
Let us take an arbitrary $w \in \tilde{\mathcal{W}}.$ Denote $z_{k+1}(w):=x_k(w)+\alpha_k(w) p_k(w)$. 
Notice that nonexpansivity of orthogonal projection and the fact that $\tilde{y} \in \Omega$ together imply  
\begin{equation}
||x_{k+1}(w)-\tilde{y}(w)||=||P_{\Omega}(z_{k+1}(w))-P_{\Omega}(\tilde{y}(w))||\leq||z_{k+1}(w)-\tilde{y}(w)||.
\label{nej_proj}
\end{equation}
 Using \eqref{qkb} and the fact that $ \bar{g}_{k} $  is subgradient of convex function $ f_{\mathcal{N}_{k}} $, i.e.,  $ \bar{g}_{k}\in  \partial f_{\mathcal{N}_{k}}(x_k) $, we have $f_{\mathcal{N}_k}(x_k)-f_{\mathcal{N}_k}(\tilde{y})\leq \overline{g}^T_k(x_k-\tilde{y})$. Dropping $w$ in order to facilitate the reading and defining 
 $$ e_k:=|f_{{\cal{N}}_k}(\tilde{y})-f(\tilde{y})|+\max_{x \in \bar{\Omega}}|f_{{\cal{N}}_k}(x)-f(x)|,$$ 
 we obtain 
\begin{eqnarray}
||z_{k+1}- \tilde{y}||^2 
&=&\nonumber
||x_{k}+\alpha_k p_k-\tilde{y}||^2 = ||x_{k}-\alpha_k \zeta_k v_k-\tilde{y}||^2 \\ \nonumber
&=& 
||x_{k}-\tilde{y}||^2-2\alpha_k \zeta_k\frac{\overline{g}^T_k}{q_k} \left(x_{k}-\tilde{y}\right) + \alpha_k^2 \zeta^2_k ||v_{k}||^2 \\ \nonumber
&\leq & 
||x_{k}-\tilde{y}||^2+2\alpha_k \frac{\zeta_k}{q_k} (f_{\mathcal{N}_{k}}(\tilde{y})-f_{\mathcal{N}_{k}}(x_k))+ \alpha_k^2 \zeta_k^2 \\ \nonumber
&\leq &
||x_{k}-\tilde{y}||^2 +2\alpha_k \frac{\zeta_k}{q_k}(f(\tilde{y})-f(x_k)+e_k) +\alpha_k^2 \zeta_k^2\\ \nonumber
&\leq&
||x_{k}-\tilde{y}||^2-2\alpha_k\frac{\zeta_k}{q_k}(f(x_k)-f(\tilde{y}))+2e_k\alpha_k \overline{\zeta}+\alpha_k^2\overline{\zeta}^2\\ \nonumber
&\leq& 
||x_{k}-\tilde{y}||^2-2 \alpha_k\frac{\underline{\zeta}}{\bar{q}}\varepsilon+2e_k\alpha_k\overline{\zeta}+\alpha_k^2\overline{\zeta}^2\\ 
&=&
||x_{k}-\tilde{y}||^2-\alpha_k \left(2 \frac{\underline{\zeta}}{\bar{q}}\varepsilon-2e_k\overline{\zeta} -\alpha_k\overline{\zeta}^2\right),
\label{dod1}
\end{eqnarray}
Since, $\lbrace x_k \rbrace \subseteq \bar{\Omega}$, ULLN under the stated assumptions implies   
 $ \lim_{k \rightarrow \infty} e_k(w) =0  $ for almost every $w \in \mathcal{W}$. Since $P(\tilde{\mathcal{W}})>0$, there must exist a sample path  $\tilde{w}\in \tilde{\mathcal{W}}$ such that $$\lim_{ k\rightarrow \infty} e_k(\tilde{w}) =0.$$ 
This further implies the existence of  $\tilde{k}(\tilde{w})\in \mathbb{N}$ such that for all $k\geq \tilde{k}(\tilde{w})$ we have 
\begin{equation}
\alpha_k (\tilde{w}) \overline{\zeta}^2+2e_k (\tilde{w}) \overline{\zeta}\leq  \varepsilon(\tilde{w}) \frac{\underline{\zeta}}{\bar{q}}
\label{nejednakost}
\end{equation}
because Step S2 of AN-SPS algorithm implies that $\lim_{k \to \infty } \alpha_k=0$ for any sample path.
Furthermore, since \eqref{dod1} holds for all $w \in \tilde{\mathcal{W}}$ and thus for $\tilde{w}$ as well, from \eqref{nej_proj}-\eqref{nejednakost} we obtain
$$||x_{k+1}(\tilde{w})-\tilde{y}(\tilde{w})||^2\leq ||z_{k+1}(\tilde{w})-\tilde{y}(\tilde{w})||^2\leq||x_{k}(\tilde{w})-\tilde{y}(\tilde{w})||^2-\alpha_k(\tilde{w}) \varepsilon(\tilde{w}) \frac{\underline{\zeta}}{\bar{q}}$$
and
$$||x_{k+s}(\tilde{w})-\tilde{y}(\tilde{w})||^2\leq ||x_{k}(\tilde{w})-\tilde{y}(\tilde{w})||^2- \varepsilon (\tilde{w})\frac{\underline{\zeta}}{\bar{q}} \sum_{j=0}^{s-1}\alpha_j(\tilde{w}).$$ 
Letting $s \to \infty$ yields a contradiction since $\sum_{k=0}^{\infty} \alpha_k \geq \sum_{k=0}^{\infty} 1 /k =\infty$ for any sample path and  we conclude that \eqref{prvideo}  holds.

Now, let us prove that \begin{equation}
\label{t1b}
\lim_{k\rightarrow\infty} x_k=x^*  \quad \text{a.s.}
\end{equation}  
Since \eqref{prvideo} holds, we know that 
\begin{equation}
\liminf_{k\rightarrow\infty} f(x_k(w))=f^*,
\label{liminf}
\end{equation}
for almost every $w \in \mathcal{W}$. In other words, there exists $\overline{\mathcal{W}} \subseteq \mathcal{W}$ such that $P(\overline{\mathcal{W}})=1$ and \eqref{liminf} holds for all $w \in \overline{\mathcal{W}}$. Let us consider arbitrary $w \in \overline{\mathcal{W}}$. We will show that $\lim_{k\rightarrow\infty} x_k(w)=x^*(w) \in X^*$ which will imply the result \eqref{t1b}.  Once again let us  drop $w$ to facilitate the notation. 
%
Let $K_1\subseteq\mathbb{N}$ be a subsequence of iterations such that
$$\lim_{k\in K_1}f(x_k)=f^*.$$
Since $\lbrace{ x_k \rbrace}_{k\in K_1}\subseteq \lbrace{ x_k \rbrace}_{k\in\mathbb{N}}$ and $\lbrace{ x_k \rbrace}_{k\in\mathbb{N}}$ is bounded, there exist $K_2\subseteq K_1$ and $\tilde{x}$ such that
\begin{equation} 
\label{rec2} \lim_{k\in K_2} x_k=\tilde{x}.
\end{equation}
Then, we have
$$f^*=\lim_{k\in K_1}f(x_k)=\lim_{k\in K_2}f(x_k)=f(\lim_{k\in K_2} x_k)=f(\tilde{x}).$$
Therefore, $f(\tilde{x})=f^*$ and we have $\tilde{x}\in X^*$. Now, we show that the whole sequence of iterates converges. 
Let $\lbrace{ x_k \rbrace}_{k\in K_2}:=\lbrace{ x_{k_i} \rbrace}_{i\in\mathbb{N}}.$ Following the steps of   \eqref{e12} and using the fact that $f(x_k)\geq f(\tilde{x})$ for all $k$, we obtain that the following holds for any $s\in\mathbb{N}$

\begin{eqnarray} 
||x_{k_i+s}- \tilde{x}||^2  
&\leq &
||x_{k_i}-\tilde{x}||^2+2\overline{\zeta}\sum_{j=0}^{s-1}\bar{e}_{k_i+j}\alpha_{k_i+j}+\overline{\zeta}^2\sum_{j=0}^{s-1}\alpha_{k_i+j}^2\\\nonumber
&\leq &
||x_{k_i}-\tilde{x}||^2+2\overline{\zeta}\sum_{j=0}^{\infty}\bar{e}_{k_i+j}\alpha_{k_i+j}+\overline{\zeta}^2\sum_{j=0}^{\infty}\alpha_{k_i+j}^2\\\nonumber
&=&
||x_{k_i}-\tilde{x}||^2+2\overline{\zeta}\sum_{j=k_i}^{\infty}\bar{e}_j\alpha_j+\overline{\zeta}^2\sum_{j=k_i}^{\infty}\alpha_j^2 =:a_i.
\end{eqnarray}
 Moreover, for any $s, m \in\mathbb{N}$ there holds 
$$||x_{k_i+s}- x_{k_i+m}||^2\leq 2 ||x_{k_i+s}- \tilde{x}||^2+2 ||x_{k_i+m}- \tilde{x}||^2\leq 4 a_i.$$
Due to the fact that  $\sum_{j=k_i}^{\infty}\bar{e}_j\alpha_j$ and $\sum_{j=k_i}^{\infty}\alpha_j^2$ are the residuals of convergent sums, and that \eqref{rec2} holds, we conclude that $$\lim_{i\to\infty}a_i=0.$$ 
Thus, we have just proved that $\{x_k\}_{k \in \mathbb{N}}$ is a Cauchy sequence and thus convergent, which together with \eqref{rec2} implies that $\lim_{k \to \infty} x_k=\tilde{x}$.
\end{proof}

\end{document}